\documentclass[12pt]{amsart}

\usepackage[top=1.6cm,bottom=2cm,left=1.1in,right=1.1in]{geometry}
\usepackage[T1]{fontenc}
\usepackage[utf8]{inputenc}
\usepackage{lmodern}
\usepackage{amsmath,amssymb,amsthm,mathtools}
\usepackage{mathrsfs}
\usepackage{enumitem}
\usepackage{microtype}
\usepackage[colorlinks=true,linkcolor=blue,citecolor=blue,urlcolor=blue]{hyperref}

\numberwithin{equation}{section}

\newtheorem{theorem}{Theorem}[section]
\newtheorem{lemma}[theorem]{Lemma}
\newtheorem{corollary}[theorem]{Corollary}
\newtheorem{proposition}[theorem]{Proposition}

\theoremstyle{remark}

\theoremstyle{definition}
\newtheorem{definition}[theorem]{Definition}
\newtheorem{setup}[theorem]{Setup}

\newcommand{\ii}{\sqrt{-1}}

\newcommand{\BC}{\mathrm{BC}}

\newcommand{\MK}{\mathcal{MK}}
\newcommand{\MN}{\mathcal{MN}}

\newcommand{\Rea}{\operatorname{Re}}
\newcommand{\Ima}{\operatorname{Im}}
\newcommand{\arccot}{\operatorname{arccot}}

\newcommand{\R}{\mathbb{R}}

\newcommand{\Span}{\operatorname{Span}}

\newcommand{\tr}{\operatorname{tr}}

\newcommand{\Cone}{\operatorname{Cone}}
\newcommand{\Dest}{\operatorname{Dest}}

\title[Boundary Cases for the supercritical LYZ equation]
{Divisorial Rigidity and Regularity of  Boundary Cases for the supercritical LYZ equation
}

\author{Jixiang Fu}
\address{Shanghai Center for Mathematical Sciences,
Fudan University,
Shanghai 200433, China}
\email{majxfu@fudan.edu.cn}

\author{Shing-Tung Yau}
\address{Yau Mathematical Sciences Center, Tsinghua University,
Beijing, 100084, China}
\email{syau@tsinghua.edu.cn}
\author{Dekai Zhang}
\address{School of Mathematical Sciences, Key Laboratory of Mathematics
and Engineering Applications (Ministry of Education), Shanghai Key Laboratory
of PMMP, East China Normal University, Shanghai 200241, China}
\email{dkzhang@math.ecnu.edu.cn}

\author{Ziyi Zhang}
\address{School of Mathematical Sciences, Fudan University,
Shanghai 200433, China}
\email{21210180101@m.fudan.edu.cn}

\date{}

\begin{document}

\begin{abstract}
We study the supercritical LYZ equation on compact K\"ahler manifolds
at the boundary of stability. Under numerical semistability, we prove
a quantitative positivity estimate on the modified nef cone. It implies
that the destabilizing prime divisors form a finite exceptional family. Assuming the existence of a smooth
semisubsolution, we prove that the associated intersection form is
negative definite on the span of classes of these divisors.
Using this rigidity, we then construct a logarithmic singular subsolution along their union
and obtain a bounded Bedford--Taylor solution which is smooth on its
complement. 
\end{abstract}

\maketitle

\section{Introduction}
\label{sec:introduction}

Let $(M,\omega)$ be a compact K\"ahler manifold of complex dimension
$n\geq2$, and let $\chi$ be a smooth closed real $(1,1)$-form.
For $u\in C^\infty(M,\mathbb R)$, write
$\chi_u=\chi+\ii\partial\bar\partial u$.
The Leung--Yau--Zaslow (LYZ) equation, also called the deformed
Hermitian--Yang--Mills equation, is
\begin{equation}\label{eq:LYZ}
 \Rea(\chi_u+\ii\omega)^n
 =\cot\theta_0\,\Ima(\chi_u+\ii\omega)^n,
\end{equation}
where the constant $\theta_0$ satisfies
\begin{equation}\label{eq:compatibility}
 \int_M\bigl(\Rea(\chi+\ii\omega)^n
 -\cot\theta_0\,\Ima(\chi+\ii\omega)^n\bigr)=0.
\end{equation}
The equation arises from Mirror symmetry
by Leung--Yau--Zaslow~\cite{LeungYauZaslow} and in the study
of supersymmetric D-branes in string theory by Mari{\~n}o--Minasian--Moore--Strominger \cite{MMMS2000}. Its analytic study on compact
K\"ahler manifolds was initiated by Jacob--Yau~\cite{JacobYau}.

Throughout the paper, $\theta_0\in(0,\pi)$ and we consider the
supercritical case
\begin{equation}\label{eq:phase}
 Q_\omega(\chi_u):=\sum_{i=1}^n\arccot\lambda_i(\chi_u)=\theta_0,
\end{equation}
where $\lambda_i(\chi_u)$ are the eigenvalues of $\chi_u$ with respect
to $\omega$ and $\arccot\lambda_i\in(0,\pi)$.
If $0<Q_\omega(\chi_u)<\pi$,
\eqref{eq:LYZ} and \eqref{eq:phase} are equivalent.
For a real $(1,1)$-form $\eta$, set
$P_\omega(\eta)=\max_{1\leq i\leq n}
\sum_{j\neq i}\arccot\lambda_j(\eta)$.
Collins--Jacob--Yau~\cite{CollinsJacobYau} obtained a priori estimates
under the existence of a $\mathcal C$-subsolution in the sense of
Sz\'ekelyhidi~\cite{Sze2018}, which in this setting is a smooth
function $\underline u$ satisfying
\begin{equation}\label{eq:C-subsolution}
 P_\omega(\chi_{\underline u})<\theta_0
 \qquad\text{on }M.
\end{equation}
They proved smooth solvability under the additional condition
$Q_\omega(\chi_{\underline u})<\pi$.
This additional condition was removed in dimension three by
Pingali~\cite{PingaliThreefolds}, in dimension four by
Lin~\cite{LinDHYM}, and in arbitrary dimension by
Lin~\cite[Theorem~1.5]{LinSolvability}.

The subsolution condition is closely related to numerical
inequalities on analytic subvarieties.
Chen~\cite[Theorem~1.7]{Chen} characterized smooth solvability by
uniform numerical stability along test families.
Chu--Lee--Takahashi~\cite[Theorem~1.3]{ChuLeeTakahashi} removed the
uniform constant and obtained a Nakai--Moishezon type criterion.

Boundary cases of the LYZ equation have been studied by both elliptic
and parabolic methods. On K\"ahler surfaces, Takahashi~\cite{TakahashiBoundary},
Chan--Jacob~\cite{ChanJacobSingularity}, and
Fu--Yau--Zhang~\cite{FuYauZhangFlow} studied the behavior of the
corresponding flows near the boundary of stability.
Murakami~\cite{MurakamiSurface,MurakamiBoundary} obtained weak
convergence results on surfaces and in higher dimensions.
Related singularity formation and canonical singular solutions were
studied by Mete~\cite{MeteCotangent} and
Datar--Mete--Song~\cite{DatarMeteSong}.

For $0\leq p\leq n$ and a fixed $\theta\in(0,\pi)$, set
\begin{equation}\label{eq:Gp}
 G_{p,\theta}(\chi)
 :=\Rea(\chi+\ii\omega)^p-\cot\theta\,\Ima(\chi+\ii\omega)^p.
\end{equation}
Thus $G_{0,\theta}=1$. We write $G_p$ when the phase is fixed.
For our first result, $\theta$ need not equal $\theta_0$.
We use the following numerical semistability condition, including
the mixed intersection inequalities.

\begin{definition}\label{def:LYZ-semistable}
We say that $(\chi,\omega,\theta)$ satisfies the
\emph{semistability condition} if, for every K\"ahler class $\gamma$,
every irreducible analytic subvariety $V\subset M$ of dimension
$1\leq p\leq n$, and every $1\leq k\leq p$,
\begin{equation}\label{eq:LYZ-subvariety-semistable}
 \int_VG_{k,\theta}(\chi)\wedge\gamma^{p-k}\geq0.
\end{equation}
\end{definition}

The geometry of subvarieties on which the numerical inequalities
fail is relevant to the degeneration of solutions.
Khalid--Sj\"ostr\"om Dyrefelt~\cite{KDS24} studied destabilizing curves
for the LYZ and $Z$-critical equations on K\"ahler surfaces, using
Zariski decomposition to obtain finiteness and negative intersection
properties. They subsequently obtained finiteness and rigidity
results for generalized Monge--Amp\`ere equations under additional
positivity assumptions~\cite{KDS26}.
In the semistability case considered here, we use the following terminology introduced by~\cite{KDS26}.

\begin{definition}\label{def:destabilizing}
Assume \eqref{eq:LYZ-subvariety-semistable}.
An irreducible analytic subvariety $V\subsetneq M$ of dimension
$1\leq p\leq n-1$ is called a \emph{destabilizing subvariety} if
\begin{equation}\label{eq:destabilizing-subvariety}
 \int_VG_{p,\theta}(\chi)=0.
\end{equation}
In particular, a prime divisor $D$ is called a
\emph{destabilizing prime divisor} if
\begin{equation}\label{eq:destabilizing-divisor}
 \int_DG_{n-1,\theta}(\chi)=0.
\end{equation}
\end{definition}

Our first result gives quantitative positivity on the modified nef
cone and proves that destabilizing prime divisors form an exceptional family in the sense of
Boucksom~\cite{Boucksom}.

\begin{theorem}\label{thm:main-modified-nef}
Assume \eqref{eq:LYZ-subvariety-semistable}.
There exists $c=c(M,\chi,\omega,\theta)>0$ such that, for every
modified nef class $\beta$,
\begin{equation}\label{eq:main-modified-nef}
 \int_MG_{n-1,\theta}(\chi)\wedge\beta
 \geq c\int_M\omega^{n-1}\wedge\beta.
\end{equation}
Moreover, the set
\begin{equation}\label{eq:Sc-intro}
 S_c:=\left\{D\text{ prime divisor}\,\middle|\,
 \int_DG_{n-1,\theta}(\chi)<c\int_D\omega^{n-1}\right\}
\end{equation}
is an exceptional family. In particular, $S_c$ is finite and
$\#S_c\leq\rho(M)$, where $\rho(M)$ is the Picard number of $M$.
\end{theorem}

We first prove \eqref{eq:main-modified-nef} for K\"ahler classes.
Yau's theorem~\cite{Yau1978} allows us to choose a background volume
form adapted to the testing class. We then solve a twisted LYZ
equation and combine a pointwise inequality for its coefficients
with the Cauchy--Schwarz inequality.
Boucksom's description of modified K\"ahler classes as pushforwards
of K\"ahler classes on modifications extends the estimate to the
modified nef cone. It follows that no nonzero positive combination
of classes of divisors in $S_c$ is modified nef.
Boucksom's linear independence theorem
\cite[Proposition~3.11(iii)]{Boucksom} gives the finiteness assertion.

As a consequence, there exists
$\delta=\delta(M,\chi,\omega,\theta)>0$ such that, for every prime
divisor $D$,
\begin{equation}\label{eq:uniform-divisor-gap-intro}
 \int_DG_{n-1,\theta}(\chi)>0
 \quad\Longrightarrow\quad
 \int_DG_{n-1,\theta}(\chi)\geq\delta\int_D\omega^{n-1}.
\end{equation}

For the remaining results, we take $\theta=\theta_0$, so that
$\int_MG_{n,\theta}(\chi)=0$.
We say that $\chi$ is a \emph{smooth semisubsolution} if
\begin{equation}\label{eq:semisubsolution}
 P_\omega(\chi)
 =\max_{1\leq i\leq n}\sum_{j\neq i}\arccot\lambda_j(\chi)
 \leq\theta\qquad\text{on }M.
\end{equation}
If the class $[\chi]$ contains such a representative, we replace
$\chi$ by that representative.
This condition and the compatibility condition imply
\eqref{eq:LYZ-subvariety-semistable}.
Moreover, $G_p(\chi)>0$ for $1\leq p\leq n-2$, so only divisors can
be destabilizing.

We next study the intersection form on the destabilizing divisor
classes. On K\"ahler surfaces, the corresponding curves have negative
self-intersection; see Khalid--Sj\"ostr\"om Dyrefelt~\cite{KDS24}.
For the $2$-Hessian equation, divisorial perturbations in the boundary
case were studied by Fu--Zhang--Zhang~\cite[Theorem~1.6 and
Section~6]{FZZ26}. For the $J$-equation,
Liu~\cite[Proposition~25]{Liu2026Boundary} proved negative definiteness
of the relevant intersection form under the smooth boundary cone
condition and an additional $J$-big assumption.
Fu--Zhang~\cite{FuZhangJBoundary} established a quantitative estimate
on the modified nef cone under numerical $J$-semistability and hence removed
this additional assumption. In the present setting, we consider the
intersection form determined by $G_{n-2,\theta}(\chi)$.

\begin{theorem}\label{strgenthen}
Assume \eqref{eq:semisubsolution}, and let $D_1,\ldots,D_N$ be all
the destabilizing prime divisors. Then the quadratic form
\[
 \gamma\longmapsto
 \int_MG_{n-2,\theta}(\chi)\wedge\gamma^2,
 \qquad \gamma\in H^{1,1}(M,\mathbb R),
\]
is negative definite on
$\Span_{\mathbb R}\{\{D_1\},\ldots,\{D_N\}\}$, where
$\{D_i\}$ denotes the cohomology class of $D_i$.
\end{theorem}

When $n=2$, this follows from the Hodge index theorem.
For $n\geq3$, approximation by smooth solutions first gives
negative semidefiniteness on the kernel of
$L(\gamma)=\int_MG_{n-1}(\chi)\wedge\gamma$.
If the divisorial intersection matrix were singular, its
nonnegative off-diagonal entries would give a nonzero effective
$\mathbb R$-divisor $E$ in its kernel.
A perturbation with both signs yields
$[G_{n-3}(\chi)]\cdot\{E\}^3=0$.
The equality case of a Hodge inequality on resolutions of the
components of $E$ then forces the pullback of $\{E\}$ to each
resolution to vanish.
The numerical criterion of Demailly--P\u{a}un~\cite{DP} makes
$\{E\}$ nef, contradicting the exceptional family property.

Set $Z=\bigcup_{i=1}^ND_i$.
If $N=0$, the numerical inequalities are strict on every proper
subvariety, and \cite[Theorem~1.3]{ChuLeeTakahashi} gives a smooth
solution. When $N>0$, negative definiteness and
\eqref{eq:uniform-divisor-gap-intro} allow us to construct a
subsolution with logarithmic singularities along $Z$.

\begin{proposition}\label{prop:intro-log-subsolution}
Assume \eqref{eq:semisubsolution} and $N>0$.
There exist positive rational numbers $a_1,\ldots,a_N$, a smooth
closed real $(1,1)$-form $\widehat\chi$, and a quasi-plurisubharmonic
function $\rho\in C^\infty(M\setminus Z)$ with logarithmic
singularities along $Z$, such that
\[
 [\widehat\chi]=[\chi]-\sum_{i=1}^Na_i\{D_i\},
 \qquad P_\omega(\widehat\chi)<\theta,
 \qquad Q_\omega(\widehat\chi)<\Theta
\]
for some $\theta<\Theta<\pi$, and
$
 \chi+\ii\partial\bar\partial\rho
 =\widehat\chi+\sum_{i=1}^Na_i[D_i]
$
in the sense of currents. Here $[D_i]$ denotes the current of
integration over $D_i$.
\end{proposition}

We use this singular subsolution to obtain a bounded solution with
local smoothness away from the destabilizing divisors.

\begin{theorem}\label{thm:regularity of the weak solution}
Assume \eqref{eq:semisubsolution}.
There exists a bounded quasi-plurisubharmonic function
$u_\infty\in L^\infty(M)\cap C^\infty(M\setminus Z)$, normalized by
$\sup_Mu_\infty=0$, such that
$\chi+\ii\partial\bar\partial u_\infty\geq\cot\theta\,\omega$
as currents and
\[
 \Rea(\chi+\ii\partial\bar\partial u_\infty+\ii\omega)^n
 =\cot\theta\,\Ima(\chi+\ii\partial\bar\partial u_\infty+\ii\omega)^n
\]
on $M$ in the Bedford--Taylor sense.
On $M\setminus Z$, the equation holds smoothly on the branch
$Q_\omega(\chi+\ii\partial\bar\partial u_\infty)=\theta$.
\end{theorem}

Sun~\cite{SunLYZBoundary} studied  supercritical LYZ
equations under boundary conditions with an additional nef and big class.
More recently, Murakami~\cite[Theorem~1.15]{MurakamiBoundary}
proved existence and uniqueness of an admissible weak solution
defined by non-pluripolar products under a boundary approximation
condition. Our regularity theorem assumes a smooth semisubsolution
and proves global boundedness and smoothness outside the union of
the destabilizing prime divisors.

The paper is organized as follows.
Section~\ref{sec:preliminaries} recalls exceptional families, the
LYZ cone, and the existence results used below.
Section~\ref{sec:destabilizing-divisors} proves
Theorem~\ref{thm:main-modified-nef} and
\eqref{eq:uniform-divisor-gap-intro}.
Section~\ref{sec:hodge} establishes the Hodge inequality and
negative semidefiniteness needed in
Section~\ref{sec:negative-definiteness}, where we prove
Theorem~\ref{strgenthen}.
Section~\ref{sec:logbar} proves
Proposition~\ref{prop:intro-log-subsolution}.
Finally, Section~\ref{sec:bounded-LYZ} proves
Theorem~\ref{thm:regularity of the weak solution} by uniform
estimates and passage to the limit.

\section{Preliminaries}~\label{sec:preliminaries}

In this section, we recall some basic facts on exceptional families and
the supercritical LYZ cone, and record a twisted solvability result for
the LYZ equation.
\subsection{Exceptional families of prime divisors}

We briefly recall
Boucksom's definition of exceptional families of prime divisors on compact complex manifolds. In this paper, we only use the K\"ahler case.

\begin{definition}
Let $M$ be a compact complex manifold with a Hermitian form $\omega_0$. A class $\alpha\in H^{1,1}_{\BC}(M,\R)$ is called
\emph{modified K\"ahler} if it contains a K\"ahler current $T$ whose
generic Lelong number $\nu(T,D)$ for every prime divisor $D$ is zero. It is called
\emph{modified nef} if, for every $\varepsilon>0$, it contains a
closed current
$
T_\varepsilon\geq-\varepsilon\omega_0
$
with $\nu(T_{\epsilon},D)=0$ for every prime divisor $D$.
\end{definition}

The set of modified K\"ahler classes is an open convex cone called the modified Kähler cone and
denoted by $\MK$. The set of modified nef classes is a closed convex cone called the modified nef cone and
denoted by $\MN$.  In this paper, $M$ is K\"ahler, so $\MK$ is nonempty and hence $\MN=\overline{\MK}$.

\begin{proposition}[Boucksom {\cite{Boucksom}}]\label{prop:Boucksom-MK}
Let $M$ be a compact complex manifold. A class $\alpha$ lies in $\MK$ if and only if there exists a
modification
\(
\mu:\widetilde M\longrightarrow M
\)
and a K\"ahler class $\widetilde\alpha$ on $\widetilde M$ such that
\(
\alpha=\mu_*\widetilde\alpha.
\)
\end{proposition}

\begin{definition}
A family  of prime divisors $D_1,\ldots,D_N$ is said
to be an \emph{exceptional} family if and only if the convex cone generated by their cohomology classes intersects the
modified nef cone $\MN$ at 0 only.
\end{definition}

Boucksom proved that the classes of prime divisors in an exceptional family are linearly
independent in $H^{1,1}_{BC}(M,\mathbb{R})$; see~\cite[Proposition~3.11(3)]{Boucksom}. Consequently,
the cardinality of an exceptional family is bounded by the Picard
number $\rho(M)$.

\subsection{The supercritical LYZ cone and twisted LYZ equation }
\label{subsec:LYZ-cone}
For $0\leq p\leq n$, recall
\begin{equation}\label{eq:Gp-section2}
G_p(\chi)
=
\Rea(\chi+\ii\omega)^p
-
\cot\theta\,\Ima(\chi+\ii\omega)^p.
\end{equation}

\begin{definition}\label{def:Gamma}
Let $\lambda_1(\chi),\ldots,\lambda_n(\chi)$ denote the eigenvalues of
$\chi$ with respect to $\omega$. We write 
\[
P_\omega(\chi)
:=
\max_i\sum_{j\neq i}\arccot\lambda_j(\chi),
\qquad
Q_\omega(\chi)
:=
\sum_j\arccot\lambda_j,
\]
For $0<\theta<\Theta<\pi$, let 
\begin{equation}\label{eq:Gamma}
\Gamma_{\omega,\theta}
:=
\left\{
\chi\in\Lambda_{\mathbb R}^{1,1}(M)
\,\middle|\,
P_\omega(\chi)<\theta
\right\},
\end{equation} and
\[
\Gamma_{\omega,\theta,\Theta}
:=
\left\{
\chi\in\Lambda_{\mathbb R}^{1,1}(M)
\,\middle|\, P_\omega(\chi)<\theta,\ Q_\omega(\chi)<\Theta\right\}.
\]
Their closures are denoted by 
$
\overline{\Gamma}_{\omega,\theta}
$  and $\overline{\Gamma}_{\omega,\theta,\Theta}
$ respectively.
\end{definition}

We record the standard property of the
supercritical cone. We refer to Lemma 8.2 of~\cite{CollinsJacobYau} for the proof.

\begin{lemma}\label{lem:descending-positivity}
If $\chi\in\Gamma_{\omega,\theta}$, then for every $x\in M$
and every complex $p$-dimensional subspace $V\subset T_xM$,
\begin{equation}\label{eq:descending-positivity}
  G_p(\chi)|_V>0,
  \qquad 1\leq p\leq n-1.
\end{equation}
If $\chi\in\overline{\Gamma}_{\omega,\theta}$, then  $G_p(\chi)|_V>0$ for $1\leq p\leq n-2$, while
$G_{n-1}(\chi)\geq0$ for every $x\in M$.
\end{lemma}

Chen proved the existence of a subsolution $\chi_{\underline{u}}\in \Gamma_{\omega,\theta,\Theta}$ under a
uniform numerical stability condition~\cite{Chen}.
Chu--Lee--Takahashi~\cite{ChuLeeTakahashi} removed the uniformity requirement and obtained
a Nakai--Moishezon type criterion. For the perturbations used in this paper, we shall also use Chen's uniform
version.

We shall use the following consequence of
\cite[Theorem~1.4]{LinSolvability}.

\begin{proposition}
\label{prop:twisted-existence}
There exists $\varepsilon_{n,\theta}>0$, depending only on $n$
and $\theta$, with the following property.
Suppose that $[\chi]$ contains a smooth form
$\chi_v\in\Gamma_{\omega,\theta}$ and that
$f\in C^\infty(M,\R)$ satisfies
\begin{equation}\label{eq:twisted-f-lower}
  f>-\varepsilon_{n,\theta}
\end{equation}
and
\begin{equation}\label{eq:twisted-compatibility}
  \int_M f\,\omega^n=\int_M G_n(\chi).
\end{equation}
Then there is a smooth representative
$\chi_u\in[\chi]\cap\Gamma_{\omega,\theta}$ satisfying
\begin{equation}\label{eq:twisted-equation}
  G_n(\chi_u)=f\,\omega^n.
\end{equation}
\end{proposition}

\begin{proof}
This is a direct corollary of~\cite[Theorem~1.4]{LinSolvability}.
For completeness, we include the details here to verify the hypotheses of
\cite[Theorem~1.4]{LinSolvability}. Let
\[
  A_{n,\theta}
  :=
  \frac{1}{
    \sin\theta\,
    \sin^{n-1}\!\left(\frac{\theta}{n-1}\right)
  },
  \qquad
  \varepsilon_{n,\theta}:=\frac12 A_{n,\theta}.
\]
For $0\leq m\leq n$, define the real polynomial
\[
  P_{m,\theta}(x)
  :=
  \Rea(x+\ii)^m-\cot\theta\,\Ima(x+\ii)^m.
\]
Then $P_{0,\theta}=1$ and
$P'_{m,\theta}=mP_{m-1,\theta}$ for $m\geq1$.
Moreover,  the largest real root of $P_{m,\theta}$ is
$\cot(\theta/m)$ for $m\geq1$.

For $a>-\varepsilon_{n,\theta}$, consider the symmetric  multi-linear
 polynomial
\[
  \mathcal F_a(\lambda_1,\ldots,\lambda_n)
  :=
  \Rea\prod_{i=1}^n(\lambda_i+\ii)
  -
  \cot\theta\,\Ima\prod_{i=1}^n(\lambda_i+\ii)
  -a,
\] 
and the polynomial
$p_a(x):=P_{n,\theta}(x)-a$.
At $x_*:=\cot(\theta/(n-1))$, we have
\[
  p_a(x_*)=-A_{n,\theta}-a<0.
\]
Since $p_a$ is strictly increasing on $(x_*,\infty)$ and
tends to $+\infty$, its largest real root  $r(p_a)$ lies in
$(x_*,\infty)$.
All positive-order derivatives of $p_a$ are independent of $a$,
and hence
\[
  r(p_a)>r(p_a')>\cdots>r(p_a^{(n-1)}).
\]
Thus $p_a$ is strictly right-Noetherian, and
$\mathcal F_a$ is strictly $\Upsilon$-stable by
\cite[Theorem~2.2]{LinSolvability}.  Its
$\Upsilon_1$-cone is exactly  $$\{(\lambda_1,\ldots,\lambda_n)\in \mathbb{R}^n \mid\max_i\sum_{j\neq i}\arccot\lambda_j<\theta\}.$$

Finally, set $\widehat\chi:=\chi-\cot\theta\,\omega$.
Then equation~\eqref{eq:twisted-equation} becomes
\[
  \widehat\chi_u^{\,n}
  =
  \sum_{k=1}^{n-2}
    b_k\,\widehat\chi_u^{\,k}\wedge\omega^{n-k}
  +(b_0+f)\omega^n,
\]
where the $b_k$ are constants depending only on $n$ and $\theta$.
This is Lin's standard form.
The preceding cone identification shows that
$\widehat\chi_v$ is a $\mathcal C$-subsolution for this equation.
Then by \cite[Theorem~1.4]{LinSolvability} and the compatibility
condition \eqref{eq:twisted-compatibility}, there exists a smooth
representative $\chi_u\in[\chi]$ satisfying
\eqref{eq:twisted-equation}.
It remains to verify that $\chi_u\in\Gamma_{\omega,\theta}$.
Let $x_0$ be a minimum point of $u-v$.
Then $\chi_u(x_0)\geq\chi_v(x_0)$, and hence
\[
P_\omega(\chi_u)(x_0)
\leq P_\omega(\chi_v)(x_0)<\theta.
\]
Suppose that $P_\omega(\chi_u)=\theta$ at some point.
Write $\theta_i=\operatorname{arccot}\lambda_i\in(0,\pi)$
for the eigenvalues $\lambda_i$ of $\chi_u$ with respect to
$\omega$ at that point, we may choose $j$ such that
$\sum_{i\neq j}\theta_j=\theta$.
Consequently,
\[
f
=
\frac{\sin\left(\theta-\sum_{i=1}^n\theta_i\right)}
     {\sin\theta\prod_{i=1}^n\sin\theta_i}
=
-\frac{1}{\sin\theta\prod_{i\neq j}\sin\theta_i}
\leq -A_{n,\theta},
\]
where the last inequality follows from the concavity of
$\log\sin$ on $(0,\pi)$.
This contradicts
$f>-\varepsilon_{n,\theta}>-A_{n,\theta}$.
Thus $P_\omega(\chi_u)\neq \theta$.
Since $M$ is connected, it follows that
$P_\omega(\chi_u)<\theta$ on $M$.
\end{proof}

\section{The destabilizing prime divisors for the supercritical LYZ equation}
\label{sec:destabilizing-divisors}

In this section, we prove Theorem~\ref{thm:main-modified-nef}.
The main ingredient is a
quantitative estimate on the modified
nef cone. 

Denote
\begin{equation}\label{eq:kappa}
\kappa_{n,\theta}
:=
\frac{\sin\theta}
{\sin^{n-1}\!\left(\theta/(n-1)\right)}
>0.
\end{equation}
Write $V_\omega:=\int_M\omega^n$.
Set
\begin{equation}\label{eq:A-B}
A:=\int_MG_n(\chi),
\qquad
B:=n\int_MG_1(\chi)\wedge\omega^{n-1}.
\end{equation}
We shall see below that $B>0$. Define
\begin{equation}\label{eq:c}
c
:=
\frac{A+\csc^2\theta\,\kappa_{n,\theta}V_\omega}{B}
=
\frac{1}{B}
\left(
A+\frac{V_\omega}
{\sin\theta\,\sin^{n-1}\!\left(\theta/(n-1)\right)}
\right)
>0.
\end{equation}
\begin{theorem}\label{thm:modified-nef}
Assume that the semistability condition
\eqref{eq:LYZ-subvariety-semistable}
holds. Then for every modified nef class $\beta$, one has
\begin{equation}\label{eq:modified-nef}
\int_MG_{n-1}(\chi)\wedge\beta
\geq
c\int_M\omega^{n-1}\wedge\beta,
\end{equation}
where $c>0$ is the constant defined in~\eqref{eq:c}.
\end{theorem}

We first prove the estimate when the testing class $\beta$ is
K\"ahler. 

\begin{lemma}\label{lem:Kahler-testing}
Under the assumptions of Theorem~\ref{thm:modified-nef}, for every
K\"ahler class $\beta$, one has
\begin{equation}\label{eq:Kahler-testing}
\int_MG_{n-1}(\chi)\wedge\beta
\geq
c\int_M\omega^{n-1}\wedge\beta.
\end{equation}
\end{lemma}

\begin{proof}
We first verify that the denominator $B$ in~\eqref{eq:A-B} is strictly
positive. Since
\[
G_1(\chi)=\chi-\cot\theta\omega,
\qquad
G_2(\chi)
=
G_1(\chi)^2-\csc^2\theta\,\omega^2,
\]
the LYZ-semistability condition, applied with the K\"ahler class
$[\omega]$, gives
\begin{equation}\label{eq:G1-estimate}
\int_MG_1(\chi)\wedge\omega^{n-1}\geq0,
\end{equation}
and
\begin{equation}\label{eq:G1-2-estimate}
\int_MG_1(\chi)^2\wedge\omega^{n-2}
\geq
\csc^2\theta\int_M\omega^n
>0,
\end{equation}
where inequality~\eqref{eq:G1-2-estimate} follows directly from
$\displaystyle\int_MG_2(\chi)\wedge\omega^{n-2}\geq0$. By the Hodge index inequality
(see, e.g.,~\cite[Chapter~VI]{Demailly}),
\[
\left(
\int_MG_1(\chi)\wedge\omega^{n-1}
\right)^2
\geq
\left(
\int_MG_1(\chi)^2\wedge\omega^{n-2}
\right)
\left(
\int_M\omega^n
\right)
>0.
\]
Together with the inequality~\eqref{eq:G1-estimate}, this yields
\begin{equation}\label{eq:B-positive}
B
=
n\int_MG_1(\chi)\wedge\omega^{n-1}
>0.
\end{equation}

Choose a K\"ahler representative of the K\"ahler class $\beta$, still denoted by $\beta$.
For $t>0$, set
\[
\chi_t:=\chi+t\omega,
\qquad
A_t:=\int_MG_n(\chi_t),
\qquad
B_t:=n\int_MG_1(\chi_t)\wedge\omega^{n-1}.
\]
By the LYZ-semistability condition,
\[
A_t
=
\sum_{k=0}^n
\binom nk
t^{n-k}
\int_MG_k(\chi)\wedge\omega^{n-k}
\geq
t^n\int_M\omega^n
>0.
\]
Similarly, if $V\subset M$ is an irreducible $p$-dimensional analytic
subvariety and $1\leq k\leq p$, then
\[
\int_VG_k(\chi_t)\wedge\omega^{p-k}
=
\sum_{j=0}^k
\binom kj
t^{k-j}
\int_VG_j(\chi)\wedge\omega^{p-j}
\geq
t^k\int_V\omega^p
>0.
\]
Thus for each fixed $t>0$, we can apply Chen's uniform criterion~\cite{Chen} to
the entire test family $(\chi+t\omega)+s\omega$, $s\geq0$. It implies that
$[\chi_t]$ contains a strict $C$-subsolution of the LYZ equation.

Define the smooth
positive function $b$ by $b\omega^n=\omega^{n-1}\wedge\beta$.
By Yau's theorem~\cite{Yau1978}, there exists a K\"ahler form $\widehat\omega\in[\omega]$
satisfying 
\begin{equation}\label{eq:testing-MA}
\widehat\omega^n
=\frac{V_\omega}{\displaystyle\int_M \omega^{n-1}\wedge \beta}\,b\omega^n.
\end{equation}
For a background form $\alpha$, we write
$G_p(\eta,\alpha):=\Rea(\eta+\ii\alpha)^p
-\text{cot}\theta\Ima(\eta+\ii\alpha)^p$;
then $G_p(\eta)=G_p(\eta,\omega)$.
All numerical integrals above are unchanged when $\omega$ is replaced
by $\widehat\omega$. In particular,
$[\chi_t+s\widehat\omega]=[\chi+(t+s)\omega]$ for every $s\geq0$,
and the same uniform strict bounds hold for this test family with
background $\widehat\omega$. Applying the uniform criterion again,
we obtain a strict $C$-subsolution in $[\chi_t]$ relative to
$\widehat\omega$.

The constant $a_t:=A_t/V_\omega$ is positive and satisfies
$\int_Ma_t\widehat\omega^n=A_t$.
Proposition~\ref{prop:twisted-existence}, applied with background
$\widehat\omega$ and right-hand side $a_t$, therefore gives
$\chi_{t,u}\in[\chi_t]\cap\Gamma_{\widehat\omega,\theta}$ such that
\begin{equation}\label{eq:mass-concentration-equation}
G_n(\chi_{t,u},\widehat\omega)
=a_t\widehat\omega^n.
\end{equation}
Fix a point and choose a $\widehat\omega$-unitary frame in which
$\chi_{t,u}$ is diagonal, with eigenvalues $\lambda_1,\ldots,\lambda_n$.
For every $i$, let
\[
S_i:=\Rea\prod_{j\neq i}(\lambda_j+\ii)
-\text{cot}\theta\Ima\prod_{j\neq i}(\lambda_j+\ii),
\qquad
T_i:=\Ima\prod_{j\neq i}(\lambda_j+\ii).
\]
If $\beta_{i\bar i}$ are the diagonal entries of $\beta$ in this frame,
then
\begin{equation}\label{eq:Gn-1-coordinate}
\frac{G_{n-1}(\chi_{t,u},\widehat\omega)\wedge\beta}
{\widehat\omega^n}
=\frac1n\sum_{i=1}^nS_i\beta_{i\bar i},
\qquad S_i>0.
\end{equation}
Expanding~\eqref{eq:mass-concentration-equation} gives
\begin{equation}\label{eq:key-pointwise-identity}
(\lambda_i-\text{cot}\theta)S_i
=a_t+\csc^2\theta\,T_i.
\end{equation}
Write $\theta_j:=\arccot\lambda_j\in(0,\pi)$.
The cone condition gives
$S_i:=\sum_{j\neq i}\theta_j<\theta$ and
\begin{equation}\label{eq:imaginary-cofactor-lower}
T_i
=\frac{\sin S_i}{\prod_{j\neq i}\sin\theta_j}
\geq\frac{\sin\theta}
{\sin^{n-1}\!\left(\theta/(n-1)\right)}
=\kappa_{n,\theta}.
\end{equation}
Indeed, let $m=n-1$. If $m=1$, both $T_i$ and
$\kappa_{n,\theta}$ equal $1$. If $m\geq2$, concavity of
$\log\sin x$ on $(0,\pi)$ yields
$\prod_{j\neq i}\sin\theta_j\leq\sin^m(S_i/m)$.
Moreover, the function $h_m(x):=\sin x/\sin^m(x/m)$ is decreasing on
$(0,\pi)$, since
$h_m'(x)/h_m(x)=\cot x-\cot(x/m)<0$.
Thus $T_i\geq h_m(S_i)\geq h_m(\theta)$, as claimed.

For every $i$, 
$\theta_i<\theta$, and hence
\(
\lambda_i>\text{cot}\theta.
\)
Consequently, $\xi_t:=\chi_{t,u}-\text{cot}\theta\widehat\omega$ is a
K\"ahler form. By~\eqref{eq:key-pointwise-identity} and
\eqref{eq:imaginary-cofactor-lower},
\[
S_i\geq
\frac{a_t+\csc^2\theta\,\kappa_{n,\theta}}
{\lambda_i-\text{cot}\theta}.
\]
It follows from~\eqref{eq:Gn-1-coordinate} that
\begin{equation}\label{eq:testing-pointwise}
G_{n-1}(\chi_{t,u},\widehat\omega)\wedge\beta
\geq
\frac{a_t+\csc^2\theta\,\kappa_{n,\theta}}{n}
\operatorname{tr}_{\xi_t}\beta\,\widehat\omega^n.
\end{equation}

Set $s_t:=\operatorname{tr}_{\omega}\xi_t>0$.
Since $\beta\leq(\operatorname{tr}_{\xi_t}\beta)\xi_t$, we have
$nb=\operatorname{tr}_{\omega}\beta
\leq s_t\operatorname{tr}_{\xi_t}\beta$.
Also, cohomological invariance gives
\[
\int_Ms_t\omega^n
=n\int_M\xi_t\wedge\omega^{n-1}
=B_t.
\]
Using~\eqref{eq:testing-MA}, \eqref{eq:testing-pointwise}, and the
Cauchy--Schwarz inequality, we obtain
\begin{equation}\label{eq:Kahler-testing-t}
\begin{aligned}
\int_MG_{n-1}(\chi_t)\wedge\beta
=&\int_MG_{n-1}(\chi_{t,u},\widehat\omega)\wedge\beta\\
\geq&
\frac{A_t+\csc^2\theta\,\kappa_{n,\theta}V_\omega}{\displaystyle\int_M \omega^{n-1}\wedge \beta}
\int_M\frac{b^2}{s_t}\,\omega^n\\
\geq&
\frac{A_t+\csc^2\theta\,\kappa_{n,\theta}V_\omega}{\displaystyle\int_M \omega^{n-1}\wedge \beta}
\frac{\left(\int_Mb\omega^n\right)^2}
{\int_Ms_t\omega^n}\\
=&
\frac{A_t+\csc^2\theta\,\kappa_{n,\theta}V_\omega}{B_t}
\int_M\omega^{n-1}\wedge\beta.
\end{aligned}
\end{equation}
Since $A_t\to A$, $B_t\to B>0$, and
$[G_{n-1}(\chi_t)]\to[G_{n-1}(\chi)]$ as $t\to0^+$,
passing to the limit in~\eqref{eq:Kahler-testing-t} yields
\eqref{eq:Kahler-testing}.
\end{proof}

\begin{proof}[Proof of Theorem~\ref{thm:modified-nef}]
Since the modified nef cone is the closure of the modified K\"ahler
cone, and both sides of~\eqref{eq:modified-nef} depend continuously
on $\beta$, it suffices to prove the estimate when $\beta$ is a
modified K\"ahler class.

By Proposition~\ref{prop:Boucksom-MK}, there exist a modification
\[
\mu:\widetilde M\longrightarrow M
\]
and a K\"ahler class $\widetilde\beta$ on $\widetilde M$ such that
$
\beta=\mu_*\widetilde\beta.$
Choose a K\"ahler form $\alpha\in\widetilde\beta$.

Fix $t>0$. The estimates in the first part of the proof of
Lemma~\ref{lem:Kahler-testing} show that $[\chi_t]$ satisfies the
strict numerical stability conditions. Hence it contains a strict
$C$-subsolution, which we denote by $\chi'_t$:
\begin{equation}\label{eq:chi-t-prime}
\chi'_t\in[\chi_t]\cap\Gamma_{\omega,\theta}.
\end{equation}
By Lemma~\ref{lem:descending-positivity},
\begin{equation}\label{eq:Gp-chi-t-prime}
G_p(\chi'_t)>0,
\qquad
1\leq p\leq n-1.
\end{equation}

For $\varepsilon_1>0$ and $\varepsilon_2>0$, define two  forms
\begin{equation}\label{eq:tilde-forms}
\widetilde\chi_{t,\varepsilon_1}
:=
\mu^*\chi'_t+\varepsilon_1\alpha,
\qquad
\widetilde\omega_{\varepsilon_2}
:=
\mu^*\omega+\varepsilon_2\alpha.
\end{equation}
 on
$\widetilde M$.

First we claim that, for every fixed $t>0$ and $\varepsilon_1>0$, if
$\varepsilon_2>0$ is sufficiently small, the pair
\(
(\widetilde\chi_{t,\varepsilon_1},
\widetilde\omega_{\varepsilon_2})
\)
satisfies the semistale condition~\eqref{eq:LYZ-subvariety-semistable}. Indeed, first set $\varepsilon_2=0$. For
every $1\leq p\leq n-1$,
\begin{equation}\label{eq:tilde-G-expansion}
G_p(\widetilde\chi_{t,\varepsilon_1},\mu^*\omega)
=
\sum_{j=0}^p
\binom pj
\varepsilon_1^j\alpha^j
\wedge
\mu^*G_{p-j}(\chi'_t,\omega)>0.
\end{equation}
Here every summand is semipositive by~\eqref{eq:Gp-chi-t-prime}, and the
last summand is
\(
\varepsilon_1^p\alpha^p>0.
\)
By openness and compactness, the same strict positivity holds after
replacing $\mu^*\omega$ by $\widetilde\omega_{\varepsilon_2}$,
provided $\varepsilon_2>0$ is sufficiently small.
That is, \begin{equation}\label{eq:tilde-form-positive}
G_p(\widetilde\chi_{t,\varepsilon_1},\widetilde\omega_{\epsilon_2})>0, \quad 1\leq p\leq n-1.
\end{equation}

Moreover, the same expansion in top degree gives
\[
\begin{aligned}
\int_{\widetilde M}
G_n(\mu^*\chi'_t+\varepsilon_1\alpha,\mu^*\omega)=A_t+
\sum_{j=1}^n\binom nj\varepsilon_1^j
\int_{\widetilde M}\alpha^j\wedge
\mu^*G_{n-j}(\chi'_t,\omega)>0.
\end{aligned}
\]
 Indeed, every term in the sum is nonnegative, and $A_t>0$.
Thus, for each fixed $\varepsilon_1>0$, after decreasing
$\varepsilon_2>0$ if necessary, we also have
\begin{equation}\label{eq:tilde-top-positive}
\int_{\widetilde M}
G_n(\widetilde\chi_{t,\varepsilon_1},
\widetilde\omega_{\varepsilon_2})>0.
\end{equation}
Thus the claim follows from \eqref{eq:tilde-form-positive} and ~\eqref{eq:tilde-top-positive}.

Therefore we can apply Lemma~\ref{lem:Kahler-testing} on
$\widetilde M$ with the K\"ahler testing class
$\widetilde\beta=[\alpha]$. We obtain
\begin{equation}\label{eq:tilde-estimate}
\int_{\widetilde M}
G_{n-1}(
\widetilde\chi_{t,\varepsilon_1},
\widetilde\omega_{\varepsilon_2}
)
\wedge\alpha
\geq
c_{t,\varepsilon_1,\varepsilon_2}
\int_{\widetilde M}
\widetilde\omega_{\varepsilon_2}^{\,n-1}
\wedge\alpha,
\end{equation}
where
\begin{equation}\label{eq:tilde-c}
c_{t,\varepsilon_1,\varepsilon_2}
=
\frac{
\displaystyle
\int_{\widetilde M}
G_n(
\widetilde\chi_{t,\varepsilon_1},
\widetilde\omega_{\varepsilon_2}
)
+
\csc^2\theta\,\kappa_{n,\theta}
\int_{\widetilde M}
\widetilde\omega_{\varepsilon_2}^{\,n}
}{
\displaystyle
n
\int_{\widetilde M}
G_1(
\widetilde\chi_{t,\varepsilon_1},
\widetilde\omega_{\varepsilon_2}
)
\wedge
\widetilde\omega_{\varepsilon_2}^{\,n-1}
}.
\end{equation}

For fixed $\epsilon_1>0$, letting $\varepsilon_2\to0^+$, and then let $\varepsilon_1\to0^+$ in~\eqref{eq:tilde-estimate}, the
continuity of intersection numbers and the projection formula give
\begin{equation}\label{eq:modified-Kahler-limit}
\begin{aligned}
\int_MG_{n-1}(\chi_t)\wedge\beta
&=
\int_{\widetilde M}
\mu^*G_{n-1}(\chi_t)\wedge\alpha\\
&\geq
c_t\int_{\widetilde M}
(\mu^*\omega)^{n-1}\wedge\alpha=
c_t\int_M\omega^{n-1}\wedge\beta,
\end{aligned}
\end{equation}
where
\begin{equation}\label{eq:c-t}
c_t
=\frac{A_t+\csc^2\theta\,\kappa_{n,\theta}V_\omega}{B_t}.
\end{equation}
Finally, letting $t\to0^+$,
we obtain
\[
\int_MG_{n-1}(\chi)\wedge\beta
\geq
c\int_M\omega^{n-1}\wedge\beta.
\]
Thus the estimate holds for every modified K\"ahler class.
Since
\(
\MN=\overline{\MK}
\) on $M$, the desired estimate for all modified
nef classes follows by continuity of integral.
\end{proof}

With the estimate developed in Theorem~\ref{thm:modified-nef}, we now
prove the exceptional family part of Theorem~\ref{thm:main-modified-nef}. Recall the set
\begin{equation}\label{eq:S}
S_c
=
\left\{
D\,\middle|\,
D\text{ is a prime divisor and }
\int_DG_{n-1}(\chi)
<
c\int_D\omega^{n-1}
\right\}.
\end{equation}

\begin{theorem}\label{thm:exceptional-family}
Assume the semistability condition~\eqref{eq:LYZ-subvariety-semistable}. Let $c>0$
be the constant in~\eqref{eq:c}
Then  $S_c$ is an exceptional family. In particular,
$S_c$ is finite and
$
\# S_c\leq \rho(M),$ where $\rho(M)$ is the Picard number of $M$.
\end{theorem}

\begin{proof}
Suppose \(D_1,\ldots,D_j\) is a family of prime divisors in $S$, and suppose the $(1,1)$-class
\[
\alpha
=
\sum_{i=1}^jx_i\{D_i\}
\neq0,
\qquad
x_i\geq0.
\]
By the definition of $S_c$,  we have
\begin{equation}\label{eq:alpha-not-nef}
\int_MG_{n-1}(\chi)\wedge\alpha
<
c\int_M\omega^{n-1}\wedge\alpha.
\end{equation}
By Theorem~\ref{thm:modified-nef}, the class $\alpha$ cannot be
modified nef. It follows that
\[
\Cone\{\{D_1\},\ldots,\{D_j\}\}\cap\MN
=
\{0\}.
\]
This proves that $D_1,\ldots,D_j$ is an
exceptional family in the sense of Boucksom.

By~\cite[Proposition~3.11(3)]{Boucksom}, the classes
\(
\{D_1\},\ldots,\{D_j\}
\)
are linearly independent in $H^{1,1}_{BC}(M,\mathbb{R})$. Hence $j$ is bounded by the Picard number of
$M$. Since the choice of family of prime divisors in  $S_c$ was arbitrary, $S$ itself is finite
and
\(
\#S\leq\rho(M).
\)
\end{proof}

The following consequence will be used later.

\begin{corollary}\label{cor:uniform-gap}
Assume the semistability condition~\eqref{eq:LYZ-subvariety-semistable}. Define the set
\begin{equation}\label{eq:Dest}
\Dest_{n-1}(\chi,\omega,\theta)
:=
\bigcup_{\substack{
D\subset M\ \text{is a prime divisor}\\
\int_DG_{n-1}(\chi)=0
}}
D.
\end{equation}
Then $\Dest_{n-1}(\chi,\omega,\theta)$ is an analytic subvariety of
$M$. Moreover, for any prime divisor $D$ satisfying $\displaystyle\int_DG_{n-1}(\chi)>0$, there exists a uniform constant $\delta>0$ independent of $D$, such that
\begin{equation}\label{eq:uniform-gap}
\int_DG_{n-1}(\chi)
\geq
\delta\int_D\omega^{n-1}.
\end{equation}
\end{corollary}

\begin{proof}
Let
\[
K
:=
\left\{
D:
D\text{ is a prime divisor and }
\int_DG_{n-1}(\chi)=0
\right\}.
\]
Then $K\subset S_c$. Hence by Theorem~\ref{thm:exceptional-family}, 
\(
K
\)
is finite, and then
\(
\Dest_{n-1}(\chi,\omega,\theta)=\bigcup_{D\in K}D\)
is an analytic subvariety.

To prove the existence of the uniform gap $\delta>0$, set
\[
S_c^+
:=
\left\{
D\in S_c:
\int_DG_{n-1}(\chi)>0
\right\}.
\]
If $S_c^+=\varnothing$, then every prime divisor $D$ with
$\int_DG_{n-1}(\chi)>0$ satisfies
\[
\int_DG_{n-1}(\chi)
\geq
c\int_D\omega^{n-1}.
\]
Thus we may take $\delta=c$. Suppose now that $S^+\neq\varnothing$. Since $S^+_c$ is finite, the
number
\[
b
:=
\min_{D\in S^+_c}
\frac{
\int_DG_{n-1}(\chi)
}{
\int_D\omega^{n-1}
}
\]
is strictly positive. Set
\(
\delta:=b>0.
\) Then $b\leq c$.
If $D\in S^+_c$, then
\[
\int_DG_{n-1}(\chi)
\geq
b\int_D\omega^{n-1}
\geq
\delta\int_D\omega^{n-1}.
\]
If $D\notin S_c^+$, then $D\notin S_c$, hence
\[
\int_DG_{n-1}(\chi)
\geq
c\int_D\omega^{n-1}
\geq
\delta\int_D\omega^{n-1}.
\]
This proves~\eqref{eq:uniform-gap}.
\end{proof}

\section{A Hodge inequality and negative semidefiniteness of an intersection form}~\label{sec:hodge}

Throughout the remainder of this paper, we assume that $(\chi,\omega)$
satisfies the semisubsolution condition. For real $(1,1)$-forms
$\alpha$ and $\beta$, write
\[
G_{p,\theta}(\alpha,\beta)
=\Rea(\alpha+\ii\beta)^p-\cot\theta\,\Ima(\alpha+\ii\beta)^p.
\]
We first prove the following Hodge inequality.

\begin{proposition}\label{prop:smooth-Hodge}
Let $(X^m,\eta)$ be a compact connected K\"ahler manifold, $m\ge2$,
and let $\theta\in(0,\pi)$. Suppose that $\alpha$ is a smooth closed
real $(1,1)$-form satisfying
\[
\alpha\in\Gamma_{\eta,\theta},\qquad
G_{m,\theta}(\alpha,\eta)\ge0.
\]
If $[\gamma]\in H^{1,1}(X,\R)$ and
$\int_XG_{m-1,\theta}(\alpha,\eta)\wedge\gamma=0$, then
\begin{equation}\label{eq:smooth-Hodge}
\int_XG_{m-2,\theta}(\alpha,\eta)\wedge\gamma^2\le0,
\end{equation}
with equality if and only if $[\gamma]=0$.
\end{proposition}

\begin{proof}
Choose a smooth closed representative $\gamma$ and set
$L=G_{m-1,\theta}(\alpha,\eta)$. Since $L$ is positive and closed,
the operator $v\mapsto L\wedge\ii\partial\bar\partial v/\eta^m$
is elliptic and self-adjoint, with kernel consisting of constants.
Thus the integral assumption gives a smooth real function $v$ such that
\[
H:=\gamma+\ii\partial\bar\partial v,\qquad L\wedge H=0.
\]
At a fixed point, choose $\eta$-unitary coordinates in which
$\alpha=\ii\sum_i\lambda_i\,dz_i\wedge d\bar z_i$, and write
$H=\ii\sum_{i,j}H_{i\bar j}\,dz_i\wedge d\bar z_j$. Set
$\theta_i=\arccot\lambda_i$ and $\delta=\theta-\sum_i\theta_i$.
The cone condition gives $-\pi<\delta<\pi$, while
\[
\frac{G_{m,\theta}(\alpha,\eta)}{\eta^m}
=\frac{\sin\delta}{\sin\theta\prod_i\sin\theta_i}\ge0.
\]
Hence $\delta\ge0$ and $\delta+\sum_i\theta_i=\theta<\pi$.
Writing $x_i=\sin\theta_i\,H_{i\bar i}$, direct expansion gives
\begin{equation}\label{eq:trig-primitive}
\sum_i\sin(\delta+\theta_i)x_i=0
\end{equation}
and
\begin{equation}\label{eq:Hodge-pointwise}
\begin{aligned}
\frac{m(m-1)}2\sin\theta\prod_k\sin\theta_k\,
\frac{G_{m-2,\theta}(\alpha,\eta)\wedge H^2}{\eta^m}
={}&\sum_{i<j}\sin(\delta+\theta_i+\theta_j)x_ix_j\\
&-\sum_{i<j}\sin\theta_i\sin\theta_j
\sin(\delta+\theta_i+\theta_j)|H_{i\bar j}|^2.
\end{aligned}
\end{equation}
All coefficients in the second sum are positive, since
$0<\delta+\theta_i+\theta_j\le\theta<\pi$.
We show that the first sum is negative unless $x=0$.
Assume $\theta_1\ge\cdots\ge\theta_m$.

If $\delta=0$, then \eqref{eq:trig-primitive} yields
\[
\sum_{i<j}\sin(\theta_i+\theta_j)x_ix_j
=-\sum_i\sin\theta_i\cos\theta_i\,x_i^2.
\]
The assertion is immediate when $\theta_1\le\pi/2$.
Otherwise, $\sum_{j\ge2}\theta_j<\pi-\theta_1<\pi/2$, so
$\sum_{j\ge2}\tan\theta_j
\le\tan(\sum_{j\ge2}\theta_j)<-\tan\theta_1$.
By Cauchy--Schwarz and \eqref{eq:trig-primitive},
\[
\sum_{j\ge2}\sin\theta_j\cos\theta_j\,x_j^2
\ge\frac{\sin^2\theta_1\,x_1^2}{\sum_{j\ge2}\tan\theta_j}
>-\sin\theta_1\cos\theta_1\,x_1^2
\]
if $x_1\ne0$; if $x_1=0$, strict positivity holds unless $x=0$.

If $\delta>0$, then \eqref{eq:trig-primitive} gives
\begin{equation}\label{eq:trig-quadratic}
2\sum_{i<j}\sin(\delta+\theta_i+\theta_j)x_ix_j
=-\frac{(\sum_i\sin\theta_i\,x_i)^2}{\sin\delta}
-\sum_i\sin(\delta+2\theta_i)x_i^2.
\end{equation}
The assertion follows unless $\sin(\delta+2\theta_1)<0$.
In this case, set $T=\sum_{j\ge2}\theta_j$; then
$0<T<(\pi-\delta)/2$. Since $\cot(2\theta_j)\ge\cot(2T)$
and $\sum_{j\ge2}\tan\theta_j\le\tan T$, we have
\[
\sum_{j\ge2}\frac{\sin^2\theta_j}{\sin(\delta+2\theta_j)}
\le\frac{\sum_{j\ge2}\tan\theta_j}
{2(\cos\delta+\sin\delta\cot(2T))}
\le\frac{\sin^2T}{\sin(\delta+2T)}.
\]
Writing $z=\sum_{j\ge2}\sin\theta_j\,x_j$, Cauchy--Schwarz gives
\[
\sum_{j\ge2}\sin(\delta+2\theta_j)x_j^2
\ge\frac{\sin(\delta+2T)}{\sin^2T}\,z^2.
\]
Consequently, the negative of the right-hand side of
\eqref{eq:trig-quadratic} is at least
\[
\frac1{\sin\delta}\left(
\sin^2(\delta+\theta_1)x_1^2+2\sin\theta_1\,x_1z
+\frac{\sin^2(\delta+T)}{\sin^2T}\,z^2\right).
\]
This quadratic form is positive definite, because
\[
\sin(\delta+\theta_1)\sin(\delta+T)-\sin\theta_1\sin T
=\sin\delta\,\sin\theta>0.
\]
If $x_1=z=0$, \eqref{eq:trig-quadratic} is strictly negative
unless all remaining $x_j$ vanish. This proves the assertion.

It follows from \eqref{eq:Hodge-pointwise} that
$G_{m-2,\theta}(\alpha,\eta)\wedge H^2\le0$, with equality
exactly when $H=0$. Since $G_{m-2,\theta}(\alpha,\eta)$ is closed,
\[
\int_XG_{m-2,\theta}(\alpha,\eta)\wedge\gamma^2
=\int_XG_{m-2,\theta}(\alpha,\eta)\wedge H^2\le0.
\]
Equality therefore implies $[\gamma]=[H]=0$, and the converse
follows from Stokes' theorem.
\end{proof}

We next consider the boundary case.
\begin{lemma}\label{lem:rotation}
Assume $\int_MG_n(\chi)=0$ and
$\chi\in\overline\Gamma_{\omega,\theta}$.  For 
small enough $\varepsilon>0$, denote 
\begin{equation}\label{eq:rotation}
\chi_\varepsilon=\cos\varepsilon\,\chi-\sin\varepsilon\,\omega,
 \quad \omega_\varepsilon=\sin\varepsilon\,\chi+
 \cos\varepsilon\,\omega,
 \quad \theta_\varepsilon=\theta+n\varepsilon.
\end{equation}
Then $\omega_\varepsilon$ is K\"ahler,
$\chi_\varepsilon\in\Gamma_{\omega_\varepsilon,\theta_\varepsilon}$,
and $[\chi_\varepsilon]$ contains a smooth form $\Omega_\varepsilon$
such that
$$G_{n,{\theta_\varepsilon}}(\Omega_\varepsilon,
\omega_\varepsilon)=0.$$
\end{lemma}
\begin{proof}
For small enough $\varepsilon$, $\omega_{\varepsilon}$ is K\"ahler.
Since $\chi_\varepsilon+\ii\omega_\varepsilon
=e^{\ii\varepsilon}(\chi+\ii\omega)$,  $\theta_{\varepsilon}$ is the principal argument of $\int_{M}(\chi_{\varepsilon}+\sqrt{-1}\omega_{\varepsilon})^n$.

At any point, choose holomorphic coordinates such that
\[
\omega=\sqrt{-1}\sum_i dz^i\wedge d\bar z^i,
\qquad
\chi=\sqrt{-1}\sum_i\lambda_i\,dz^i\wedge d\bar z^i.
\]
Then the eigenvalues of $\chi_\varepsilon$ with respect to
$\omega_\varepsilon$ are
\[
\lambda_i^\varepsilon
=
\frac{\lambda_i\cos\varepsilon-\sin\varepsilon}
{\lambda_i\sin\varepsilon+\cos\varepsilon}
=
\cot(\theta_i+\varepsilon),
\qquad
\theta_i=\arccot\lambda_i.
\]
Hence $\theta_i^\varepsilon=\theta_i+\varepsilon$. Since
$\chi\in\overline{\Gamma}_{\omega,\theta}$, for any $1\le i\le n$,
\[
\sum_{j\ne i}\theta_j^\varepsilon
\le
\theta+(n-1)\varepsilon
<
\theta+n\varepsilon
=
\theta_\varepsilon,
\]
and thus
$
\chi_\varepsilon\in
\Gamma_{\omega_\varepsilon,\theta_\varepsilon}.$ By Lin's result, there exists a smooth function $u^{\varepsilon}$ such that $\chi_{\varepsilon, u_{\varepsilon}}$ solving the perturbed supercritical  LYZ equation.
\end{proof}

\begin{theorem}\label{thm:seminegativity}
Assume $\int_MG_n(\chi)=0$ and
$\chi\in\overline\Gamma_{\omega,\theta}$.  Define
\[
 L(e)=[G_{n-1}(\chi)]\cdot e,\qquad
 Q(e,f)=[G_{n-2}(\chi)]\cdot e\cdot f,
\]
for $e,f\in H^{1,1}(M,\mathbb{R})$.
Then $Q\le0$ on $\ker L$.  In particular, if
$D_1,\ldots,D_N$ are the destabilizing prime divisors, the matrix
\begin{equation}\label{eq:Q-matrix}
 Q=(q_{ij}),\qquad q_{ij}=[G_{n-2}(\chi)]\cdot D_i\cdot D_j,
\end{equation}
is negative semidefinite and $q_{ij}\ge0$ for $i\ne j$.
\end{theorem}

\begin{proof}
Since $\chi\in\overline\Gamma_{\omega,\theta}$, $(\chi,\omega)$ satisfies the semistability condition~\eqref{eq:LYZ-subvariety-semistable}, then by
Theorem~\ref{thm:modified-nef}, we have $L([\omega])>0$.  For
$e\in\ker L$, let $\chi_{\varepsilon,u_{\varepsilon}}$ be as above and denote by
\[
 b_\varepsilon=
 \frac{[G_{n-1,\theta_\varepsilon}(\chi_{\varepsilon,u_{\varepsilon}},
 \omega_\varepsilon)]\cdot e}
 {[G_{{n-1},{\theta_\varepsilon}}(\chi_{\varepsilon,u_{\varepsilon}},
 \omega_\varepsilon)]\cdot[\omega]}.
\]
  We have
$b_\varepsilon\to0$, and by definition
\[
 [G_{n-1,\theta_\varepsilon}(\chi_{\varepsilon,u_{\varepsilon}},
 \omega_\varepsilon)]\cdot(e-b_\varepsilon[\omega])=0.
\]
Proposition~\ref{prop:smooth-Hodge}, applied to
$e-b_\varepsilon[\omega]$, gives a nonpositive second intersection
for every $\varepsilon$.  
Passing to the limit gives $Q(e,e)\le0$.  Since each $\{D_i\}\in
\ker L$, so the 
matrix is negative semidefinite.
At last, since $G_{n-2}(\chi)>0$ in the strong sense, and the intersection
of two distinct prime divisors is an effective codimension two
cycle, hence $q_{ij}\ge0$ for $i\ne j$.
\end{proof}

\section{Divisorial rigidity and negative definiteness}~\label{sec:negative-definiteness}

Assume $(\chi,\omega)$ satisfies the semisubsolution condition.  
Recall $D_1,\ldots,D_N$ be all destabilizing prime divisors. They form a
finite exceptional family.
Recall that 
\begin{equation}\label{eq:section5-LQ}
  L(e)=[G_{n-1}(\chi)]\cdot e,
  \qquad
  Q(e,f)=[G_{n-2}(\chi)]\cdot e\cdot f.
\end{equation} for $e,f\in H^{1,1}(M,\mathbb{R})$.
By Theorem~\ref{thm:seminegativity}, $Q$ is seminegative definite on $\ker L$. 

We prove that $Q$ is negative definite on $\operatorname{Span}_{\mathbb{R}}\{\{D_1\},\ldots,\{D_N\}\}\subset\ker L$. When $M$ is a K\"ahler surface, this follows from usual Hodge index theorem. So  we assume $n\geq 3$ in this section.

\subsection{Hodge LYZ inequality on the resolution}
We first extend the equality case of
Proposition~\ref{prop:smooth-Hodge} to possibly singular prime divisors.  Motivated by~\cite{Liu2026Boundary}, we solve a twisted LYZ equation on the resolution. Here we establish a weighted
Hodge integral estimate, and we don't need to prove the higher order estimate of the solution to the equation on the resolution.

For each destabilizing prime divisor $D$, we work under the following setup.
\begin{setup}\label{setup:resolution}
Let $D$ be a destabilizing prime divisor, and let
\[
 h\colon Y=D^\nu\longrightarrow D\hookrightarrow M
\]
be the normalization map followed by the inclusion.
Choose a projective log resolution $f\colon X\to Y$
which is an isomorphism over $Y_{\mathrm{reg}}$ and such that
\( E:=\bigl(f^{-1}(Y_{\mathrm{sing}})\bigr)_{\mathrm{red}}
\)
is a simple normal crossings (SNC) divisor.
Fix a K\"ahler form $\kappa$ on $X$, and let
\[
 m:=n-1,\qquad
 g:=h\circ f,\qquad
 A_0:=g^*\chi,\qquad
 B_0:=g^*\omega.
\]
\end{setup}

\begin{lemma}
\label{lem:restriction-regularisation}
There exist constants $\ell<K$ such
that, for
\begin{equation}\label{eq:explicit-regularisation}
  A_t=A_0+Kt\kappa,
  \qquad
  B_t=B_0+t\kappa,
  \qquad t>0,
\end{equation}
$B_t$ is Kähler, the eigenvalues of $A_t$ with respect to $B_t$ lie in $[\ell,K]$, and
\begin{equation}\label{eq:regularised-phase-bounds}
  Q_{B_t}(A_t)\leq\theta,
  \qquad
  P_{B_t}(A_t)\leq\theta-\operatorname{arccot}K<\theta.
\end{equation}
Consequently,
$G_{m,\theta}(A_t,B_t)=b_tB_t^m$ for a smooth function $b_t\geq0$, and there
is a constant $c>0$, independent of $t>0$, such that every real
$(1,1)$-form $H$ satisfies
\begin{equation}\label{eq:Bt-weighted-coercivity}
  G_{m-1,\theta}(A_t,B_t)\wedge H=0
  \quad\Longrightarrow\quad
  -G_{m-2,\theta}(A_t,B_t)\wedge H^2
  \geq c|H|_{B_t}^2B_t^m.
\end{equation}
\end{lemma}

\begin{proof}
Set
\(  V:=g^{-1}(D_{\mathrm{reg}})\subset X.
\)
The normalization is an isomorphism over $D_{\mathrm{reg}}$, and $f$
is an isomorphism over $Y_{\mathrm{reg}}$. Hence
$g|_V\colon V\to D_{\mathrm{reg}}$ is biholomorphic. In particular,
$V$ is dense in $X$ and $B_0$ is positive definite on $V$.

Fix $x\in V$ and let $y=g(x)$. Let
$\nu_1\leq\cdots\leq\nu_n$ be the eigenvalues of $\chi$ with respect
to $\omega$ at $y$, and let $\mu_1\leq\cdots\leq\mu_m$ be the
eigenvalues of their restrictions to $T_yD_{\mathrm{reg}}$. Then
\[
  \nu_j\leq\mu_j\leq\nu_{j+1},
  \qquad 1\leq j\leq m=n-1.
\]
The semisubsolution condition on $M$  yields
\[
  Q_{B_0}(A_0)(x)
  =\sum_{j=1}^{m}\operatorname{arccot}\mu_j
  \leq\sum_{j=1}^{n-1}\operatorname{arccot}\nu_j
  \leq\theta,\quad x\in V.
\]

By compactness of $M$, we can choose constants $\ell<K$ such that
\[
  \ell\omega\leq\chi<K\omega,
  \qquad
  K>\max\{0,\ell,\cot(\theta/m)\},
\]
and define $A_t,B_t$ by \eqref{eq:explicit-regularisation}.
Since $B_0\geq0$, the form $B_t=B_0+t\kappa$ is K\"ahler for every
$t>0$. Moreover,
\[
  A_t-\ell B_t
  =g^*(\chi-\ell\omega)+(K-\ell)t\kappa\geq0,
  \qquad
  KB_t-A_t=g^*(K\omega-\chi)\geq0.
\]
Thus all eigenvalues of $A_t$ with respect to $B_t$ lie in the fixed
interval $[\ell,K]$.

We next compute the eigenvalues on $V$.
For $x\in V$ and $0\neq v\in T_xX$, write
$R_t(v)=A_t(v,\bar v)/B_t(v,\bar v)$, including $t=0$.
A direct calculation gives
\[
  R_t(v)-R_0(v)
  =
  \frac{
    t\kappa(v,\bar v)
    \bigl(KB_0(v,\bar v)-A_0(v,\bar v)\bigr)
  }{
    B_0(v,\bar v)
    \bigl(B_0(v,\bar v)+t\kappa(v,\bar v)\bigr)
  }
  \geq0.
\]
Let $\lambda_1(t,x)\leq\cdots\leq\lambda_m(t,x)$ denote the ordered
 eigenvalues of $A_t$  with respect to $B_t$ at $x$.
The min--max principle, applied to the same subspaces of
$T_xX$, implies
\[
  \lambda_j(t,x)\geq\lambda_j(0,x),
  \qquad 1\leq j\leq m.
\]
Consequently,
\[
  Q_{B_t}(A_t)(x)
  \leq Q_{B_0}(A_0)(x)
  \leq\theta,
  \qquad x\in V.
\]
For each fixed $t>0$, the function $Q_{B_t}(A_t)$ is continuous
on all of $X$. Since $V$ is dense, the same upper bound holds
throughout $X$. The eigenvalue upper bound gives $\operatorname{arccot}\lambda_j(t,x)\geq\arccot K$, so
\[
  P_{B_t}(A_t)
  =Q_{B_t}(A_t)-\min_j\operatorname{arccot}\lambda_j(t,x)
  \leq\theta-\arccot K<\theta.
\]
This proves \eqref{eq:regularised-phase-bounds}.
Write $\theta_j=\arccot\lambda_j(x,t)$. Then
\[
  b_t:=\frac{G_{m,\theta}(A_t,B_t)}{B_t^m}
  =
  \frac{\sin\bigl(\theta-Q_{B_t}(A_t)\bigr)}
       {\sin\theta\prod_{j=1}^{m}\sin\theta_j(t,x)}\geq0.
\]

Now we will apply 
Proposition~\ref{prop:smooth-Hodge}.
To make the independence of $t$ explicit, consider the compact set
\[
  \mathcal C
  :=
  \left\{
    \lambda\in[\ell,K]^m:
    \sum_{j=1}^{m}\operatorname{arccot}\lambda_j\leq\theta
  \right\}.
\]
On $\mathbb T^m$, let
\[
  \beta=\sqrt{-1}\sum_{j=1}^{m}dz^j\wedge d\bar z^j,
  \qquad
  A_\lambda=\sqrt{-1}\sum_{j=1}^{m}
                  \lambda_j\,dz^j\wedge d\bar z^j.
\]
For every $\lambda\in\mathcal C$, the preceding phase calculation gives
\[
  P_\beta(A_\lambda)\leq\theta-\arccot K<\theta,
  \qquad
  G_{m,\theta}(A_\lambda,\beta)\geq0.
\]
Identify constant real $(1,1)$-forms on $\mathbb T^m$ with Hermitian
matrices, and consider the normalised primitive pairs
\[
  \mathcal S
  :=
  \left\{
    (\lambda,H):
    \lambda\in\mathcal C,\quad |H|_\beta=1,\quad
    G_{m-1,\theta}(A_\lambda,\beta)\wedge H=0
  \right\}.
\]
This is a compact set.
The function
\[
  (\lambda,H)\longmapsto
  -\frac{G_{m-2,\theta}(A_\lambda,\beta)\wedge H^2}{\beta^m}
\]
is continuous, and it is strictly positive on $\mathcal S$ due to Proposition~\ref{prop:smooth-Hodge}.
It therefore has a positive minimum
$c=c(m,\theta,\ell,K)>0$.

At any $x\in X$ and for any $t>0$, choose a $B_t$-unitary frame
diagonalising $A_t$. Its eigenvalues belong to $\mathcal C$.
Applying the preceding minimum estimate in this situation and using
homogeneity of $H$ gives
\[
  G_{m-1,\theta}(A_t,B_t)\wedge H=0
  \quad\Longrightarrow\quad
  -G_{m-2,\theta}(A_t,B_t)\wedge H^2
  \geq c|H|_{B_t}^2B_t^m.
\]
This is \eqref{eq:Bt-weighted-coercivity}, with the same constant
for every $x\in X$ and $t>0$.
\end{proof}

\begin{lemma}[$O(t)$ weighted energy]\label{lem:weighted-energy}
In the notation above, suppose that $H_t$  are smooth real $(1,1)$-forms, satisfying
\begin{equation}\label{eq:weighted-energy-assumption}
  \int_X|H_t|_{B_t}^2B_t^m\leq Ct.
\end{equation}
Let
\[
  S=\{y\in Y_{\mathrm{reg}}:\operatorname{rank}dh_y\leq m-2\},
  \qquad
  U_0=f^{-1}(Y_{\mathrm{reg}}),
  \qquad
  U=f^{-1}(Y_{\mathrm{reg}}\setminus S).
\]
Then   $V=g^{-1}(D_{\mathrm{reg}})\subset U\subset U_0$, $S$ is an analytic subset on $Y_{\mathrm{reg}}$ of complex codimension at least two, and for every compactly supported smooth
$(2m-2)$-form $\zeta$ on $U$,
$\int_UH_t\wedge\zeta\to0$ as $t\to 0$.
\end{lemma}

\begin{proof}
Since the normalization is an isomorphism over $D_{\mathrm{reg}}$,
we have $h^{-1}(D_{\mathrm{reg}})\subset Y_{\mathrm{reg}}$ and
$\operatorname{rank}dh=m$ on this set, hence $h^{-1}(D_{\mathrm{reg}})\cap Z=\emptyset$.
Thus $V\subset U\subset U_0$.

The set $S$ is a closed analytic subset of $Y_{\mathrm{reg}}$,
locally defined by the $(m-1)\times(m-1)$ minors of $dh$.
If an irreducible component $Z\subset S$ had dimension
$r\geq m-1$, then $h|_{Z_{\mathrm{reg}}}$ would have finite fibres
and hence generic rank $r$.
This contradicts with $\operatorname{rank}dh_y\leq m-2$ on $Z$.
Thus $\dim_{\mathbb C}S\leq m-2$.

On a compact set $K_0\subset\subset U$, let
$0\leq\mu_1\leq\cdots\leq\mu_m$ be the eigenvalues of $B_0$ with respect to $\kappa$. By the definition of $S$, $\mu_2$ has a positive lower bound on $K_0$, hence for
$0<t\leq1$ one has
\begin{equation}\label{eq:weighted-local-comparison}
  B_t^m\geq c_{K_0}(\mu_1+t)\kappa^m,
  \qquad
  |H|_{B_t}^2\geq c_{K_0}|H|_\kappa^2.
\end{equation}
The second inequality uses only the upper bound $B_t\leq C_{K_0}\kappa$. Combining
\eqref{eq:weighted-local-comparison} with
\eqref{eq:weighted-energy-assumption} gives
\[
  \int_{K_0}(\mu_1+t)|H_t|_\kappa^2\kappa^m\leq C_{K_0}t.
\]
If $K_0$ contains the support of $\zeta$, weighted Cauchy--Schwarz inequality yields
\[
  \left|\int_UH_t\wedge\zeta\right|^2
  \leq C_{K_0,\zeta}t\int_{K_0}\frac{\kappa^m}{\mu_1+t}
  =C_{K_0,\zeta}\int_{K_0}\frac{t}{\mu_1+t}\,\kappa^m.
\]
The finite map $h$ is generically immersive, so $\mu_1>0$ almost everywhere.
The last integrand is bounded and converges almost everywhere to
zero. Dominated convergence theorem proves the assertion.
\end{proof}

\begin{lemma}[Exceptional support]\label{lem:exceptional-support}
Let $f\colon X\to Y$ be as above.  Recall that the reduced inverse image $E=f^{-1}(Y_{\text{sing}})$ is an SNC divisor, write $E=\sum_{\alpha} E_{\alpha}$, where $E_{\alpha}$ are its distinct irreducible components.. If $\xi\in H^2(X,\mathbb R)$ restricts
to zero on $X\setminus E$, then
\[
  \xi=\sum_\alpha c_\alpha\{E_\alpha\}.
\]
If, in addition, $\xi=g^*\zeta$, where $g=h\circ f$ and $\zeta$ is a
smooth closed form on $M$, then $\xi=0$.
\end{lemma}

\begin{proof}
Consider the cohomology exact sequence: 
\[
  H_E^2(X,\mathbb R)\longrightarrow H^2(X,\mathbb R)
  \longrightarrow H^2(X\setminus E,\mathbb R).
\]
The Thom isomorphism and Mayer--Vietoris for the SNC union identify the
image of the first map with the span of the divisor classes $\{E_\alpha\}$;
multiple intersections have real codimension at least four and do not
contribute in degree two. Hence $\xi=\{F\}$ for an $\mathbb R$-Cartier divisor
$F=\sum_\alpha c_\alpha E_\alpha$ on $X$.
Since $Y$ is normal and
$\operatorname{codim}_{\mathbb C}Y_{\mathrm{sing}}\geq2$, every $E_\alpha$ is $f$-exceptional, and hence $F$ is $f$-exceptional. For every curve
$C$ contracted by $f$,
\[
  F\cdot C=\int_C g^*\zeta=0.
\]
 Hence both $F$ and $-F$ are $f$-nef.
The negativity lemma
\cite[Lemma~3.6.2(2)]{BCHM}, in the complex analytic setting explained
in \cite[Section~11, after Definition~11.1]{FujinoMMP},
applied over Stein neighbourhoods in $Y$, gives
$F\leq0$ and $-F\leq0$. Thus $F=0$, and hence $\xi=\{F\}=0$.
\end{proof}

For a smooth closed real $(1,1)$-form $\zeta$ on $M$, define
\begin{equation}\label{eq:LD-QD}
  L_D(\zeta)=\int_DG_{n-2}(\chi)\wedge\zeta,
  \qquad
  Q_D(\zeta)=\int_DG_{n-3}(\chi)\wedge\zeta^2.
\end{equation}
Both numbers may equivalently be computed on $X$ using $A_0,B_0$ and
$H_0=g^*\zeta$.
\begin{theorem}[Singular LYZ Hodge inequality]
\label{thm:singular-LYZ-Hodge}
Assume $(\chi,\omega)$ satisfies the semisubsolution condition. For every destabilizing prime divisor $D$,
\[
  L_D(\zeta)=0\quad\Longrightarrow\quad Q_D(\zeta)\leq0.
\]
If equality holds, then $g^*[\zeta]=0$ in $H^2(X,\mathbb R)$.
\end{theorem}

\begin{proof}
Recall $m=n-1$. Use Lemma~\ref{lem:restriction-regularisation} and let
\[
  L_t=G_{m-1,\theta}(A_t,B_t),
  \qquad
  R_t=G_{m-2,\theta}(A_t,B_t),
  \qquad
  H_0=g^*\zeta.
\]
Set
\[
  a_t=\frac{\int_XL_t\wedge H_0}{\int_XL_t\wedge B_t},
  \qquad
  \xi_t=[H_0]-a_t[B_t].
\]
When $t\to 0$,  
\begin{equation}\label{eq:limiting-denominator}
  \int_XL_t\wedge B_t\to\int_XG_{m-1,\theta}(A_0,B_0)\wedge B_0
  =\int_DG_{n-2}(\chi)\wedge\omega>0,
\end{equation}
while $\displaystyle\int_X L_t\wedge H_0$ at $t=0$ is $L_D(\zeta)=0$. It follows that
\begin{equation}\label{eq:at-order}
  a_t=O(t).
\end{equation} 
Since $\displaystyle\int_X L_t\wedge \xi_t=0$, Proposition~\ref{prop:smooth-Hodge} gives $\displaystyle\int_XR_t\wedge\xi_t^2\leq0$. Passing to the limit
gives precisely $Q_D(\zeta)\leq0$.

Suppose now that $Q_D(\zeta)=0$. Then by continuity and ~\eqref{eq:at-order},
\begin{equation}\label{eq:Et-order}
  0\leq\mathcal E_t:=-\int_XR_t\wedge\xi_t^2\leq Ct.
\end{equation}
For every fixed $t>0$, solve the equation
\[
  L_t\wedge\sqrt{-1}\,\partial\bar\partial v_t
  =-L_t\wedge(H_0-a_tB_t),
\]
Its Fredholm compatibility condition is exactly the definition of $a_t$. Denote $H_t=H_0-a_tB_t+\sqrt{-1}\,\partial\bar\partial v_t.$
Since $L_t$ and $R_t$ are closed, Stokes' theorem and
\eqref{eq:Bt-weighted-coercivity} give
\[
  \mathcal E_t=-\int_XR_t\wedge H_t^2
  \geq c\int_X|H_t|_{B_t}^2B_t^m.
\]
Lemma~\ref{lem:weighted-energy} now shows that, for every compactly
supported closed $(2m-2)$-form $\Psi$ on $U$,
\[
  \int_UH_0\wedge\Psi
  =\int_UH_t\wedge\Psi+a_t\int_UB_t\wedge\Psi
  \longrightarrow0.
\]
Here the
first term tends to zero by Lemma~\ref{lem:weighted-energy}, while the
second tends to zero by \eqref{eq:at-order}, since $B_t$ is uniformly
bounded on the fixed compact support of $\Psi$.

It remains to prove that $[H_0]=0$. The pairing
\[
  H^2(U,\mathbb R)\times H_c^{2m-2}(U,\mathbb R)
  \longrightarrow\mathbb R,
  \qquad
  ([\alpha],[\Psi])\longmapsto\int_U\alpha\wedge\Psi,
\]
is nondegenerate by Poincaré duality with compact supports on the smooth
oriented manifold $U$. The preceding vanishing for every compactly
supported closed $\Psi$ therefore gives
\begin{equation}\label{eq:vanishing-on-U}
  [H_0]|_U=0\qquad\text{in }H^2(U,\mathbb R).
\end{equation}

Let $Z:=U_0\setminus U=f^{-1}(S)$. Since $f$ is an isomorphism over
$Y_{\mathrm{reg}}$, $Z$ has complex codimension at least two in $U_0$.
Excision theorem and the Thom isomorphism implies that $H_Z^2(U_0,\mathbb R)=0$. Consider the exact sequence of cohomology groups:
\[
  H_Z^2(U_0,\mathbb R)
  \longrightarrow H^2(U_0,\mathbb R)
  \longrightarrow H^2(U,\mathbb R),
\]
it implies that the restriction map $H^2(U_0,\mathbb R)\to H^2(U,\mathbb R)$ is
injective. The preceding vanishing~\eqref{eq:vanishing-on-U} hence implies
\[
  [H_0]|_{U_0}=0 \qquad\text{in }H^2(U_0,\mathbb R).
\]
Finally, $U_0=X\setminus E$. Since 
$[H_0]=g^*[\zeta]$,
Lemma~\ref{lem:exceptional-support} implies that 
$[H_0]=0$. 
\end{proof}

\subsection{Negative definiteness}

We now prove that $Q$ is negative definite on the span of 
$\{D_1\},\ldots,\{D_N\}$.
 We record the following lemma, which gives a criterion for the nefness of an effective $\mathbb{R}$-divisor. For an effective Cartier divisor on projective manifolds, nefness of its
normal bundle implies nefness of the divisor; see
\cite[Example~1.4.6]{LazarsfeldPositivityI}.
The analogous statement in the K\"ahler setting follows
from \cite[Proposition~3.3(iv)]{DP}, applied to the
current of integration over the divisor. We give a proof of the following formulation using resolutions of the
components, based on
\cite[Theorem~4.3(iii)]{DP}; see also the proof in~\cite[Proposition~25]{Liu2026Boundary} for the case $\{E\}^2=0$.

\begin{lemma}
\label{lem:effective-nefness}
Let $E=\sum_i a_iD_i$ be a nonzero effective $\mathbb{R}$-divisor on a compact
Kähler manifold $M$. For each component $D_i$, let
\(
  h_i\colon D_i^\nu\to D_i\hookrightarrow M
\)
be the normalization and let
\(
  f_i\colon \widetilde D_i\to D_i^\nu
\)
be a resolution. Set $g_i=h_i\circ f_i: \widetilde D_i\to M$.
If the pullback  
\(
  g_i^*\{E\}\in H^{1,1}(\widetilde D_i,\mathbb R)
\)
is nef class on $\widetilde D_i$ for every $i$, then $\{E\}$ is nef on $M$.
\end{lemma}

\begin{proof}
We use Demailly--Păun's numerical criterion
\cite[Theorem~4.3(iii)]{DP} for nef class. It suffices to prove that, for any $1\leq p\leq n$, any irreducible analytic subvariety $V\subset M$ of dimension $p$, and any Kähler class $\beta$, $\displaystyle\int_V \{E\}\wedge \beta^{p-1}\geq0$.

For $p=n$,  $V=M$ and $\displaystyle\int_M \{E\}\wedge\beta^{n-1}\geq0$ since $E$ is effective divisor.
For $p\leq n-1$, if $V\not\subset\operatorname{Supp}E$,  the
restriction of $E$ on $V$ is effective divisor and it gives
$\displaystyle\int_V \{E\}\wedge\beta^{p-1}\geq0$.  If $V\subset\operatorname{Supp}E$, since \(V\) is irreducible, there exists a component \(D_i\)
of \(\operatorname{Supp}(E)\) such that \(V\subset D_i\). Choose an
irreducible component $W$ of $g_i^{-1}(V)\subset\widetilde D_i$ that dominates $V$ and choose a Kähler
form $\kappa$ on $\widetilde D_i$. The support theorem gives
\[
  (g_i)_*\bigl([W]\wedge\kappa^{\dim W-p}\bigr)=c[V],
\]  
  where 
  \[c\int_V\beta^p
  =\int_W g_i^*\beta^p\wedge\kappa^{\dim W-p}>0,
\] see~\cite[Chapter~III, \S2, Corollary~2.14]{Demailly} for instance. 
The
projection formula and nefness of $g_i^*\{E\}$ yield
\[
  c\int_V\{E\}\wedge\beta^{p-1}
  =\int_Wg_i^*\{E\}\wedge g_i^*\beta^{p-1}
       \wedge\kappa^{\dim W-p}\geq0.
\]
These
inequalities hold for every irreducible $V$ and every Kähler class $\beta$.
Thus the Demailly--Păun numerical criterion
\cite[Theorem~4.3(iii)]{DP} proves that $\{E\}$ is nef.
\end{proof}

We also need the following lemma. 
\begin{lemma}\label{lem:cubic-rigidity}
Assume $(\chi,\omega)$ satisfies the semisubsolution condition. If $e\in H^{1,1}(M,\mathbb R)$ satisfies $L(e)=Q(e,e)=0$, then
\begin{equation}\label{eq:cubic-rigidity}
  [G_{n-3}(\chi)]\cdot e^3=0.
\end{equation}
\end{lemma}

\begin{proof}
If $f\in\ker L$,  by the seminegative definiteness of $Q$ on $\ker L$, we have  
\[
  Q(e+rf,e+rf)=2rQ(e,f)+r^2Q(f,f)\leq 0, \quad r\in\mathbb{R}.
\]
It implies that $Q(e,f)=0$. In particular, $Q(e,D_i)=0$ for every $i$, since
$D_i\in\ker L$.

Choose a smooth representative $\eta\in e$ and set
\begin{equation}\label{eq:chi-s-cubic}
  \chi_s=\chi-s\eta+Ks^2\omega,
  \qquad 0<|s|\ll 1,
\end{equation}
where $K>0$ is a sufficiently large constant that will be chosen later.

For $1\leq p\leq n-2$, $G_p(\chi)|_V>0$ for any $x\in M$  and any $p$-dimensional subspace $V\subset T_x M$ by Lemma ~\ref{lem:descending-positivity}.  By compactness of $M$, for $0<|s|<<1$,
\begin{equation}\label{eq:Gp-strong-positive}
    G_p(\chi_s)|_V>0,\quad 1\leq p\leq n-2, \quad V\subset T_x M.
    \end{equation}

 For a
destabilizing prime divisor $D_i$,  a direct expansion gives
\begin{equation}\label{eq:null-divisor-expansion}
\begin{split}
  \int_{D_i}G_{n-1}(\chi_s)
  & =-(n-1)sQ(e,D_i)+(n-1)s^2
  \left(KA_i+\frac{n-2}{2}B_i\right)+O(|s|^3),
\end{split}
\end{equation}
where
\[
  A_i=\int_{D_i}G_{n-2}(\chi)\wedge\omega>0,
  \qquad
  B_i=[G_{n-3}(\chi)]\cdot e^2\cdot D_i.
\]
The linear term in \eqref{eq:null-divisor-expansion} vanishes. Since every $A_i$ is positive, one may choose a large
$K$ such that $KA_i+\frac{n-2}{2}B_i>0$ for $1\leq i\leq N$. Then $\displaystyle\int_{D_i}G_{n-1}(\chi_s)>0$
for sufficiently small $0<|s|<<1$.

For every other prime divisor $D'$, by Corollary~\ref{cor:uniform-gap}, there exists a uniform gap $\delta>0$ independent of $D'$, such that
\[
  \int_{D'}G_{n-1}(\chi_s)
  \geq (\delta-C|s|)\int_{D'}\omega^{n-1}>0.
\]
In conclusion, for any prime divisor $D$, we have 
\begin{equation}\label{eq:along-divisor}
      \int_{D}G_{n-1}(\chi_s)>\delta'\int_M \omega^{n-1},
      \end{equation}
where $\delta'$ is independent of $D$. Clearly $L([\omega])>0$; see Theorem~\ref{thm:modified-nef} for instance. By the assumptions $L(e)=Q(e,e)=0$, the top-degree expansion is
\begin{equation}\label{eq:top-expansion-cubic}
  \int_MG_n(\chi_s)=nKL([\omega])s^2+O(|s|^3)>0.
\end{equation}
Moreover,
\begin{equation}\label{eq:G_n-1-w}
[G_{n-1}(\chi_s)]\cdot[\omega]>L([\omega])/2>0, \quad\text{for }  0<|s|<<1.
\end{equation}

Consider the test family $\chi_s+r\omega$, $r\geq0$.
By \eqref{eq:Gp-strong-positive}, ~\eqref{eq:along-divisor},
\eqref{eq:top-expansion-cubic} and \eqref{eq:G_n-1-w},
there exists $\varepsilon_s>0$, independent of $r$ and $V$,
such that, for every $1\leq p\leq n$ and every
$p$-dimensional subvariety $V\subset M$,
\begin{equation}\label{eq:Gp-test-family-expansion}
\begin{aligned}
 \int_VG_p(\chi_s+r\omega)
 &=\sum_{j=0}^p\binom pj r^j
   \int_VG_{p-j}(\chi_s)\wedge\omega^j\geq(n-p)\varepsilon_s\int_V\omega^p.
\end{aligned}
\end{equation}
Thus the  test family satisfies the
uniform stable criterion of Chen~\cite[Proposition~5.2]{Chen}. Then Chen's criterion gives a strict subsolution in $[\chi_s]$, and a smooth $\widehat\chi_s\in[\chi_s]$ satisfying
$G_n(\widehat\chi_s)=b_s\omega^n$ with $b_s>0$.

Let
\[
  c_s=\frac{[G_{n-1}(\chi_s)]\cdot e}
             {[G_{n-1}(\chi_s)]\cdot[\omega]},
  \qquad
  e_s=e-c_s[\omega].
\]
Then $[G_{n-1}(\widehat \chi_s)]\cdot e_s=[G_{n-1}(\chi_s)]\cdot e_s=0$. We apply Proposition~\ref{prop:smooth-Hodge} to $(\widehat\chi_s, \omega,e_s)$, then it yields
$[G_{n-2}(\chi_s)]\cdot e_s^2\leq 0$. Since $[G_{n-1}(\chi_s)]\cdot e$ at $s=0$ is $L(e)=0$, and its first derivative is
$-(n-1)Q(e,e)=0$, we have $c_s=O(s^2)$. A direct expansion gives
\[
  [G_{n-2}(\chi_s)]\cdot e_s^2
  =-(n-2)s[G_{n-3}(\chi)]\cdot e^3+O(s^2)\leq0.
\]
Taking first $s>0$ and then $s<0$ proves
\eqref{eq:cubic-rigidity}. 
\end{proof}

 We now prove that the intersection matrix $Q$ is negative definite.
This is equivalent to the assertion of Theorem~\ref{strgenthen}.
\begin{theorem}[Negative definiteness]
\label{thm:negative-definiteness}
Assume $(\chi,\omega)$ satisfies the semisubsolution condition.
Let $D_1,\ldots,D_N$ be all the destabilizing prime
divisors. Then the $N\times N$ real matrix
\[
  Q=(q_{ij}),
  \qquad  
  q_{ij}=[G_{n-2}(\chi)]\cdot[D_i]\cdot[D_j],
\]
is negative definite.
\end{theorem}

\begin{proof}
By Theorem~\ref{thm:seminegativity}, the matrix $Q$ is negative semidefinite,
and $q_{ij}\geq0$ for $i\neq j$.
We argue by contradiction, assuming that $Q$ is  not negative definite, then $Q$ is a singular matrix.

Then there exists a vector $0\neq x$ such that $Qx=0$. Set
$a=(|x_1|,\ldots,|x_N|)^T$.
Since $q_{ij}\geq 0$, we have
\[
  a^TQa-x^TQx
  =
  2\sum_{i<j}q_{ij}
  \bigl(|x_i||x_j|-x_ix_j\bigr)
  \geq0.
\]
Since $x^TQx=0$ and $Q$ is negative semidefinite,
we obtain $a^TQa=0$, and hence $Qa=0$.

Write $a=(a_1,\ldots,a_N)^T.$ Without loss of generality, suppose that
$a_i>0$ for $1\leq i\leq r$ and $a_i=0$ for $i>r$.
Set
\[
  E=\sum_{i=1}^r a_iD_i,
  \qquad
  \{E\}=\sum_{i=1}^r a_i\{D_i\}.
\]
The divisor $E$ is nonzero and effective. Recall 
every destabilizing divisor satisfies $L(\{D_i\})=0$. Then
\[
  L(\{E\})=0,
  \qquad
  Q(\{E\},\{D_j\})
  =\sum_{i=1}^r a_iq_{ij}
  =(Qa)_j=0
  \quad(1\leq j\leq N).
\]
In particular,
$Q(\{E\},\{E\})=0$. Then by Lemma~\ref{lem:cubic-rigidity}, we have
\begin{equation}\label{eq:cubic-E-zero}
  0=[G_{n-3}(\chi)]\cdot \{E\}^3
   =\sum_{i=1}^r a_i\int_{D_i}G_{n-3}(\chi)\wedge\{E\}^2.
\end{equation}
On the other hand, since $\displaystyle\int_{D_i} G_{n-2}(\chi)\wedge \{E\}=Q(\{E\},\{D_j\})=0$, hence by Theorem~\ref{thm:singular-LYZ-Hodge} applied on $D_i$,  we have \begin{equation}\label{eq:Gn-3-seminegative}
 \int_{D_i}G_{n-3}(\chi)\wedge\{E\}^2
  \leq0,
  \qquad 1\leq i\leq r.
\end{equation}
 Since  $a_i>0$ for $1\leq i\leq r$, by~\eqref{eq:cubic-E-zero} and~\eqref{eq:Gn-3-seminegative},  $\displaystyle\int_{D_i}G_{n-3}(\chi)\wedge\{E\}^2=0$ for  $1\leq i\leq r$.

For each $1\leq i\leq r$, let
$g_i\colon X_i\to D^\nu\to D\hookrightarrow M$ be the map obtained
from a resolution of the normalization of $D_i$ constructed in Setup~\ref{setup:resolution}.
Applying the equality case of
Theorem~\ref{thm:singular-LYZ-Hodge} to $\{E\}$  gives
\[
  g_i^*\{E\}=0
  \quad\text{in }H^2(X_i,\mathbb R),
  \qquad 1\leq i\leq r.
\]
Then by Lemma~\ref{lem:effective-nefness}, $\{E\}$ is a nef class on $M$, and hence modified nef on $M$. Then 
$$\{E\}\in
\Cone\{\{D_1\},\ldots,\{D_j\}\}\cap\MN. $$
By Theorem~\ref{thm:exceptional-family}, $\{E\}=0$, and  $a_i=0$ for $1\leq i\leq N$ since $\{D_i\}$ are linearly independent, so $a=0$, which is a contradiction. Therefore $Q$ is negative
definite.
\end{proof}

The following corollary is a key ingredient for the perturbation to obtain a stable class by subtracting the corresponding
effective divisor $E$ from $[\chi]$.
\begin{corollary}
\label{cor:common-divisorial-direction}
If $N>0$, there are positive rational numbers $a_i$ such that
$E=\sum_i a_iD_i$ satisfies
\begin{equation}\label{eq:common-direction}
  [G_{n-2}(\chi)]\cdot \{E\}\cdot \{D_j\}<0,
  \qquad 1\leq j\leq N.
\end{equation}
\end{corollary}

\begin{proof}
Since $Q$ is negative definite ,we can solve $Qa=-\mathbf 1$ for a vector $a=(a_1,\ldots,a_N)\in \mathbb{R}^N$. Write $a=a^+-a^-$ with $a^+,a^-$ are positive part and negative part respectively. since $q_{ij}\geq0$ for $i\neq j$, we have
$\langle Qa^+,a^-\rangle\geq0$. If $a^-\neq0$, negative definiteness gives
$-\langle Qa^-,a^-\rangle>0$, hence
$\langle Qa,a^-\rangle>0$, contradicting
$\langle-\mathbf1,a^-\rangle<0$. Thus $a\geq0$ and it's not hard to see $a_i>0$ for $1\leq i\leq N$. Set $E=\sum_i a_iD_i$, then it satisfies ~\eqref{eq:common-direction}. 
 By approximation, we can choose all $a_i$ are rational numbers.
\end{proof}

\section{A logarithmic singular subsolution}
\label{sec:logbar}

Throughout this section, we assume $(\chi,\omega)$ satisfies the semisubsolution condition.
We construct a logarithmic singular subsolution by subtracting a small
effective divisor $E$ from $[\chi]$. The resulting class satisfies strict
inequalities on all proper subvarieties, whereas its top-dimensional
integral is negative. Thus the numerical criterion of
\cite[Theorem~1.3]{ChuLeeTakahashi} does not apply directly.
We follow the continuity argument of
\cite[Section~7]{ChuLeeTakahashi}, based on
\cite[Section~5]{Chen}, and explain the modifications
needed to allow a negative normalizing constant.

\subsection{Perturbation of the class}

Suppose first that $N>0$. Consider the divisor
$E=\sum_{i=1}^N a_iD_i$ given in Corollary~\ref{cor:common-divisorial-direction}, with $a_i\in\mathbb Q_{>0}$.
Then $L(E)=0$ and $Q(E,D_i)<0$ for every $i$.
Choose smooth Hermitian metrics $h_i$ on $\mathcal O(D_i)$, and denote
their normalized curvature forms by $\theta_i$. Let $s_i$ be the canonical section of $\mathcal O(D_i)$
with zero divisor $D_i$. The Poincar\'e--Lelong formula gives
\begin{equation}
\label{eq:logbar-PL}
 \frac{\ii}{2\pi}\partial\bar\partial
 \log\|s_i\|_{h_i}^2
 =[D_i]-\theta_i
\end{equation}
in the sense of currents.
Set $\theta_E=\sum_{i=1}^N a_i\theta_i$
and $\chi_s=\chi-s\theta_E$.

\begin{lemma}
\label{lem:logbar-perturbation}
For  sufficiently small $s>0$,  for $1\leq p\leq n-2$, one has
$
 G_p(\chi_s)|_V>0$ for any $x\in M$ and any $p$-dimensional subspace $V\subset T_xM$,
 $\displaystyle\int_DG_{n-1}(\chi_s)>0$ for every prime divisor $D$,
and $\displaystyle\int_MG_{n-1}(\chi_s)\wedge\omega>0$. Moreover,
\begin{equation}
\label{eq:logbar-negative-defect}
 \int_MG_n(\chi_s)=\binom n2s^2Q(E,E)+O(s^3)<0.
\end{equation}
\end{lemma}

\begin{proof}
For $1\leq p\leq n-2$, $G_p(\chi_s)|_V>0$ follows from
Lemma~\ref{lem:descending-positivity} and compactness. For any destabilizing prime divisor $D_i$, we have
$\displaystyle\int_{D_i}G_{n-1}(\chi_s)=-(n-1)sQ(E,D_i)+O(s^2)>0$.
For the other prime divisors, the uniform gap in
Corollary~\ref{cor:uniform-gap} gives
\[
 \int_DG_{n-1}(\chi_s)\geq(\delta-Cs)\int_D\omega^{n-1}>0,
\]
with $C$ independent of $D$.
$\displaystyle\int_MG_{n-1}(\chi_s)\wedge\omega>0$ at $s=0$ by
Theorem~\ref{thm:modified-nef}, and remains positive for small $s$.
Finally, the constant and linear terms in expansion of $\displaystyle\int_MG_n(\chi_s)$
vanish, while
$Q(E,E)=\sum_i a_iQ(E,D_i)<0$. This proves
\eqref{eq:logbar-negative-defect}.
\end{proof}

 Chen's twisted existence theorem
\cite[Proposition~5.5]{Chen}   dealt with the case that right-hand side of the LYZ equation may change sign, while its integral is
nonnegative. We therefore use Lemma~\ref{lem:logbar-twisted-phase} below
for the equations with negative integral. The other change is to obtain
the positive volume estimate needed for mass concentration without
assuming positivity of the right-hand side.

\subsection{The twisted equation and a volume estimate}

For clarity, we recall
$G_{d,\theta}(A,B)=\Rea(A+\ii B)^d-\cot\theta\,\Ima(A+\ii B)^d$. We record the twisted existence statement.

\begin{lemma}
\label{lem:logbar-twisted-phase}
Let $(Y^d,\omega)$ be a compact K\"ahler manifold, $d\geq2$, and let
$0<\theta<\Theta<\pi$, with $\Theta-\theta<\theta/(d-1)$.
There is $\epsilon(d,\theta,\Theta)>0$, depending only on these
parameters, with the following property. Suppose that $[\eta]$ contains
a smooth form $\eta_v$ with $P_\omega(\eta_v)<\theta$, and that
\[
 f>-\epsilon(d,\theta,\Theta),\qquad
 \int_Yf\omega^d=\int_YG_{d,\theta}(\eta,\omega),
 \qquad f\in C^\infty(Y,\mathbb R).
\]
Then there is a smooth representative $\eta_u\in[\eta]$ satisfying
\begin{equation}
\label{eq:logbar-twisted-phase}
 G_{d,\theta}(\eta_u,\omega)=f\omega^d,\qquad
 P_\omega(\eta_u)<\theta,\qquad Q_\omega(\eta_u)<\Theta.
\end{equation}
\end{lemma}

\begin{proof}
Let $\varepsilon_{d,\theta}$ be the constant in
Proposition~\ref{prop:twisted-existence}. Choose
\begin{equation}
\label{eq:logbar-phase-threshold}
 \epsilon(d,\theta,\Theta)=\frac12\min\left\{
 \varepsilon_{d,\theta},\,
 \frac{\min\{\sin(\Theta-\theta),\sin(\theta/(d-1))\}}{\sin\theta}
 \right\}.
\end{equation}
Proposition~\ref{prop:twisted-existence} gives a solution with $P_\omega(\eta_u)<\theta$.
We only need to check the bound for $Q_\omega$.
Write the eigenvalues as $\lambda_i=\cot\theta_i$, with
$\theta_i\in(0,\pi)$, and let  $\delta= Q_\omega(\eta_u)-\theta$.
The partial-phase inequalities imply
$-\theta<\delta<\theta/(d-1)<\pi$. The equation can be written as
\begin{equation}
\label{eq:logbar-phase-identity}
 f=-\frac{\sin\delta}{\sin\theta\prod_i\sin\theta_i}.
\end{equation}
If $f\geq0$, then $ Q_\omega(\eta_u)\leq\theta$. If $f<0$ and $ Q_\omega(\eta_u)\geq\Theta$, then
\[
 -f\geq
 \frac{\min\{\sin(\Theta-\theta),\sin(\theta/(d-1))\}}{\sin\theta},
\]
contrary to \eqref{eq:logbar-phase-threshold}. Thus $ Q_\omega(\eta_u)<\Theta$.
\end{proof}

We also need the following pointwise lemma to obtain the positive volume lower bound needed for mass concentration,
allowing the normalizing constant to be negative.

\begin{lemma}
\label{lem:logbar-volume}
Let $A$ and $B$ be real $(1,1)$-forms on a complex vector space
of dimension $d$, with $B>0$. Fix
$0<\theta<\Theta_1<\Theta_2<\pi$, and assume that
$Q_B(A)<\Theta_1$. Set
\[
 \Psi:=A-\cot\Theta_2\,B,
 \qquad a:=\cot\Theta_1-\cot\Theta_2>0.
\]
Then there exists $C=C(d,\theta,\Theta_1,\Theta_2)>0$ such that
\[
 \Psi\geq aB,\qquad |G_{d,\theta}(A,B)|\leq C\Psi^d.
\]
Moreover, if
\[
 G_{d,\theta}(A,B)=(R-1+c)B^d,
 \qquad R\geq0,\qquad c\geq-c_0
\]
for some $c_0\geq0$, then
\begin{equation}
\label{eq:logbar-positive-concentration}
 \Psi^d\geq bR B^d,
 \qquad b:=\bigl(C+(1+c_0)a^{-d}\bigr)^{-1}>0.
\end{equation}
In particular, $b$ depends only on
$d,\theta,\Theta_1,\Theta_2$ and $c_0$.
\end{lemma}

\begin{proof}
Choose a $B$-unitary frame in which $A$ is diagonal, with
eigenvalues $\lambda_1,\ldots,\lambda_d$.
$Q_B(A)<\Theta_1$ implies
$\arccot\lambda_i<\Theta_1$, and hence
$\lambda_i>\cot\Theta_1$ for every $i$.
Thus $\Psi\geq aB$.
Since $\lambda_i-\cot\Theta_2\geq a$, we have
\[
 \sqrt{1+\lambda_i^2}
 \leq \left(1+\frac{\sqrt{1+\cot\Theta_2^2}}{a}\right)(\lambda_i-\cot\Theta_2).
\]
Using
$|\operatorname{Re}z-\cot\theta\,\operatorname{Im}z|
\leq |z|/\sin\theta$, we obtain
\[
 \frac{|G_d^\theta(A,B)|}{B^d}
 \leq \frac{1}{\sin\theta}
       \prod_{i=1}^d\sqrt{1+\lambda_i^2}
 \leq C\prod_{i=1}^d(\lambda_i-\cot\Theta_2)
 =C\frac{\Psi^d}{B^d},
\]
where one may take
\[
 C=\frac{1}{\sin\theta}
   \left(1+\frac{\sqrt{1+\cot\Theta_2^2}}{a}\right)^d.
\]

Finally, set $D:=\Psi^d/B^d$. The inequality $\Psi\geq aB$
gives $D\geq a^d$. Therefore
\[
 R=\frac{G_d^\theta(A,B)}{B^d}+1-c
 \leq CD+1+c_0
 \leq \bigl(C+(1+c_0)a^{-d}\bigr)D,
\]
which proves \eqref{eq:logbar-positive-concentration}.
\end{proof}

\subsection{A numerical criterion for the twisted equation}

Let $\eta_{t,0}$, $t\geq0$, be a smooth test family from
$\eta$: $\eta_{0,0}=\eta$,
$\eta_{t_2,0}-\eta_{t_1,0}>0$ for $t_2>t_1$, and
$\eta_{t,0}>\cot(\theta/n)\omega$ for all sufficiently large $t$.
We record the following lemma, which is a consequence of the proof of
\cite[Theorem~3.1]{ChuLeeTakahashi}.

\begin{lemma}
\label{lem:logbar-local-subsolution}
Suppose that, for every $t\geq0$ and every proper irreducible analytic
subvariety $W\subsetneq M$ of dimension $p\geq1$,
\begin{equation}
\label{eq:logbar-proper-stability}
 \int_WG_{p,\theta}(\eta_{t,0},\omega)>0.
\end{equation}
Then, for every $t\geq0$ and proper analytic subset $Z\subsetneq M$,
there is a neighbourhood $U_Z$ and $v\in C^\infty(U_Z)$ such that
\[
 Q_\omega(\eta_{t,0}+\ii\partial\bar\partial v)<\theta
 \quad\text{on }U_Z.
\]
\end{lemma}

\begin{proof}
The proof of \cite[Theorem~3.1]{ChuLeeTakahashi} applies using only
\eqref{eq:logbar-proper-stability}. Indeed, for each subvariety of
dimension $p<n$, the positive  integral required in
\cite[Lemma~3.3]{ChuLeeTakahashi} is its own $p$-dimensional integral,
which is strictly positive by hypothesis. The induction, the equations
on the resolutions, and the construction in Sections~4--6 of that paper
therefore remain unchanged. The extension to a neighbourhood is given
in the proof of its Theorem~3.1 in Section~7. None of these steps uses
the  integral  $\displaystyle\int_M G_n(\eta_{t,0},\omega)$, although its nonnegativity is
included in the statement of that theorem. To obtain the assertion at
a given $t$, apply the same proof to the family
$\eta_{t+s,0}$, $s\geq0$.
\end{proof}

\begin{theorem}
\label{thm:logbar-signed-criterion}
Let $n\geq2$ and $0<\theta<\pi$. There are constants
$\epsilon_{n,\theta}>0$ and
$\Theta_0\in(\theta,\pi)$, depending only on $n,\theta$, such that the
following holds. Define
\[
 \tau=\frac{\int_M G_{n,\theta}(\eta,\omega)}{\int_M\omega^n},
 \qquad \tau>-\epsilon_{n,\theta}.
\]
Then the following conditions are equivalent:
\begin{enumerate}[label=\textup{(\roman*)},leftmargin=*,itemsep=2pt]
\item For every test family from $\eta$, one has
\eqref{eq:logbar-proper-stability} and
\begin{equation}
\label{eq:logbar-ambient-stability}
 \int_MG_{n,\theta}(\eta_{t,0},\omega)
 \geq\tau\int_M\omega^n\qquad(t\geq0).
\end{equation}
\item These inequalities hold along one such test family.
\item There is a smooth representative $\eta_u\in[\eta]$
satisfying
\begin{equation}
\label{eq:logbar-signed-solution}
 G_{n,\theta}(\eta_u,\omega)=\tau\omega^n,
 \qquad \eta_u\in\Gamma_{\omega,\theta,\Theta_0}.
\end{equation}
\end{enumerate}
\end{theorem}

\begin{proof}
The implication (i)$\Rightarrow$(ii) is immediate.
For (iii)$\Rightarrow$(i), let $\alpha_t=\eta_{t,0}-\eta\geq0$.
The representative $\widehat\eta+\alpha_t$ is a strict subsolution,
so \eqref{eq:logbar-proper-stability} holds.
Moreover, $\dot\alpha_t\geq0$ and
\[
 \frac{d}{dt}\int_MG_{n,\theta}(\widehat\eta+\alpha_t,\omega)
 =n\int_M\dot\alpha_t\wedge
 G_{n-1,\theta}(\widehat\eta+\alpha_t,\omega)\geq0.
\]
This proves \eqref{eq:logbar-ambient-stability}.

For (ii)$\Rightarrow$(iii), fix a test family as in (ii). We follow the
proof of~\cite[Proposition 5.2]{Chen} and ~\cite[Theorem~1.3]{ChuLeeTakahashi}.

We first consider a continuity path.
Choose the phase constants as follows:
\begin{equation}
\label{eq:logbar-phase-parameters}
 \begin{gathered}
 \sigma=\min\left\{\frac{\theta}{16n},
 \frac{\pi-\theta}{16(n+1)},\frac\pi{32}\right\},\qquad
 \Theta_0=\theta+\sigma,\quad\Theta=\theta+2\sigma,\\
K=\cot(2\sigma),\quad
 \widetilde\theta=\theta+2n\sigma,\quad
 \Theta_1=\widetilde\theta+\sigma,\quad \Theta_2=\widetilde\theta+2\sigma.
 \end{gathered}
\end{equation}
They lie in $(0,\pi)$ and satisfy
$\Theta_0-\theta<\theta/(n-1)$ and
$\Theta_1-\widetilde\theta<\widetilde\theta/(2n-1)$.
Let
\begin{equation}
\label{eq:logbar-numerical-threshold}
 \begin{split}
 C_K&=\frac{\sin\theta}{\sin\widetilde\theta}(1+K^2)^{n/2}, \qquad \epsilon_{n,\theta}
 =\frac18\min\left\{
 \epsilon(n,\theta,\Theta_0),\,
 C_K^{-1}\epsilon(2n,\widetilde\theta,\Theta_1)\right\}.
 \end{split}
\end{equation}
For $t\geq0$, write
$c_t=\int_MG_{n,\theta}(\eta_{t,0},\omega)/\int_M\omega^n$.
Thus $c_t\geq\tau>-\epsilon_{n,\theta}$.
Let $I$ be the set of $t\subset[0,+\infty]$ for which
\[
 G_{n,\theta}(\eta_t,\omega)=c_t\omega^n,
 \qquad \eta_t\in[\eta_{t,0}]\cap\Gamma_{\omega,\theta,\Theta_0}
\]
has a smooth solution. Lemma~\ref{lem:logbar-twisted-phase} shows that
$I$ contains all sufficiently large $t$ and is relatively open.
If $s\in I$, then $[s,+\infty)\subset I$, since for any $s'>s$, adding
$\eta_{s',0}-\eta_{s,0}>0$ to a strict subsolution at $s$ gives one at
$s'$. Let $t_*=\inf I$.

We then construct a equation on the product.
For $t\in I$ close to $t_*$, consider $X=M\times M$ and
\[
 \Omega=\pi_1^*\omega+\pi_2^*\omega,\qquad
 \widetilde\eta_t=\pi_1^*\eta_t+K\pi_2^*\omega.
\]
By our choice, $\Theta_0+2(n-1)\sigma=\widetilde\theta-\sigma$, hence $P_\Omega(\widetilde\eta_t)<\widetilde\theta$.
Also $Q_\Omega(\widetilde\eta_t)<\Theta_1$.
Let $\Delta\subset X=M\times M$ be the diagonal. Choose a finite
open cover $\{U_j\}$ of $X$, local holomorphic generators
$\{f_{j,k}\}_k$ of the ideal of $\Delta$ on $U_j$, and a 
smooth partition of unity $\{\rho_j\}$ subordinate to $\{U_j\}$. Set
\[
 \psi_r=\log\left(\sum_j\rho_j\sum_k|f_{j,k}|^2+r^2\right),
 \qquad \sigma_{r,\ell}=\Omega+\ell^2\ii\partial\bar\partial\psi_r\in[\Omega].
\]
Choose $0<\beta<1$ so small that
$(1-\beta)^{2n}-1>-\epsilon(2n,\widetilde\theta,\Theta_1)/4$.
Then choose $\ell>0$ sufficiently small that
$\sigma_{r,\ell}\geq(1-\beta)\Omega$ for all sufficiently small $r$.
This positive $\ell$ is fixed independently of $t$.
Define
\begin{equation}
\label{eq:logbar-product-rhs}
 R_{r,\ell}=\frac{\sigma_{r,\ell}^{2n}}{\Omega^{2n}},\qquad
 F_{t,r}=R_{r,\ell}-1+C_Kc_t.
\end{equation}
The identity $K+\ii=(1+K^2)^{1/2}e^{2\ii\sigma}$ gives
\begin{equation}
\label{eq:logbar-product-compatibility}
 \begin{split}
 \int_XG_{2n}^{\widetilde\theta}(\widetilde\eta_t,\Omega)
 &=\binom{2n}{n}C_K
 \left(\int_MG_{n,\theta}(\eta_t,\omega)\right)
 \left(\int_M\omega^n\right)\\
 &=C_Kc_t\int_X\Omega^{2n}=\int_XF_{t,r}\Omega^{2n}.
 \end{split}
\end{equation}
Furthermore,
\[
 F_{t,r}\geq(1-\beta)^{2n}-1-C_K\epsilon_{n,\theta}
 >-\tfrac12\epsilon(2n,\widetilde\theta,\Theta_1).
\]
When $c_t<0$, the integral in
\eqref{eq:logbar-product-compatibility} is negative. Thanks to Lemma~\ref{lem:logbar-twisted-phase}, we can obtain a smooth closed
form $\eta_{t,r}\in[\widetilde\eta_t]$ satisfying
\begin{equation}
\label{eq:logbar-product-equation}
 G_{2n}^{\widetilde\theta}(\eta_{t,r},\Omega)=F_{t,r}\Omega^{2n},\qquad
 P_\Omega(\eta_{t,r})<\widetilde\theta,\quad Q_\Omega(\eta_{t,r})<\Theta_1.
\end{equation}

Apply Lemma~\ref{lem:logbar-volume} with $d=2n$,
$\theta=\widetilde\theta$ and $c_0=C_K\epsilon_{n,\theta}$.
There are constants $a,b>0$, independent of $t,r$, such that
\begin{equation}
\label{eq:logbar-uniform-product-mass}
 \Psi_{t,r}:=\eta_{t,r}-\cot \Theta_2\,\Omega\geq a\Omega,\qquad
 \Psi_{t,r}^{2n}\geq b\sigma_{r,\ell}^{2n}.
\end{equation}
This replaces the estimate obtained from a positive right-hand side
in \cite[equation~(5.2)]{ChuLeeTakahashi}. In particular, both the
concentration parameter and the lower bound in
\eqref{eq:logbar-uniform-product-mass} remain fixed as $t\downarrow t_*$.

We now apply the mass concentration, truncation and gluing arguments of ~\cite[Section 5]{Chen} and ~\cite[Sections~5--7]{ChuLeeTakahashi}. We specify the estimates that
allow those arguments to be used here.

Choose $C_0>\max\{0,-\cot\Theta_2\}$ and let
$H_{t,r}=\eta_{t,r}+C_0\Omega$. Then $H_{t,r}$ are K\"ahler forms
also satisfying \eqref{eq:logbar-uniform-product-mass}
with $\Psi_{t,r}$ replaced by $H_{t,r}$. Since
\(
  [H_{t,r}]
  =\pi_1^*[\eta_{t,0}]
   +K\pi_2^*[\omega]+C_0[\Omega],
\) we have
\[
  \int_{M\times M} H_{t,r}^k\wedge\Omega^{2n-k}\leq C_k,
  \qquad 1\leq k\leq n,
\]
where $C_k$ is independent of $r$ and of $t\in I$
sufficiently close to $t_*$. Fix such a $t$.
By weak compactness of positive currents, there exists a
subsequence $r_j\downarrow0$ such that
\[
  H_{t,r_j}^k\rightharpoonup\Xi_{t,k},
  \qquad 1\leq k\leq n,
\]
where each $\Xi_{t,k}$ is a closed positive $(k,k)$-current.
By a similar argument as~\cite[p.~596]{Chen} or~\cite[Lemma~5.1]{ChuLeeTakahashi},
using \eqref{eq:logbar-uniform-product-mass}, we obtain
\[
 \Xi_{t,n}\geq d_0[\Delta],\qquad d_0>0,
\]
where $d_0$ is independent of $t$; and for $k\leq n-1$, $\Xi_{t,k}\wedge \Omega^{n-k}$ have no
mass on $\Delta$. For the $n$th power, we have the decomposition
\[
\begin{aligned}
 \eta_{t,r}^{n}
 ={}&H_{t,r}^{n}
 +\sum_{\substack{2\leq j\leq n\\ j\ \mathrm{even}}}
   \binom nj C_0^j H_{t,r}^{n-j}\wedge \Omega^j-\sum_{\substack{1\leq j\leq n\\ j\ \mathrm{odd}}}
   \binom nj C_0^j H_{t,r}^{n-j}\wedge \Omega^j.
\end{aligned}
\]
Thus, if $\Upsilon_{t,n}$ denotes the weak limit of
$\eta_{t,r}^{n}$, then
\(
 \mathbf{1}_{\Delta}\Upsilon_{t,n}
 =\mathbf{1}_{\Delta}\Xi_{t,n}
 \geq d_0[\Delta].
\)

We consider the fiber integral 
\begin{equation}
\label{eq:logbar-fibre-integral}
 T_{t,r}=\frac{1}{(n+1)V_K}(\pi_1)_*
 \Ima(\eta_{t,r}+\ii\pi_2^*\omega)^{n+1},\qquad
 V_K=\Ima(K+\ii)^n\int_M\omega^n>0.
\end{equation}
Then $T_{t,r}\in[\eta_{t,0}]$. By
\cite[Lemma~5.12 and Proposition~5.13]{Chen}, it satisfies
$P_\omega(T_{t,r})<\theta$ and $Q_\omega(T_{t,r})<\Theta$.
These estimates use only
the second and third conditions in \eqref{eq:logbar-product-equation}.

Apply the truncation estimates
\cite[Lemma~5.5, Corollary~5.6 and Lemmas~5.7--5.11]{ChuLeeTakahashi}
in the form for both $P_\omega$ and $Q_\omega$ described in
~\cite[Section~7]{ChuLeeTakahashi}. 
The uniform diagonal mass and the phase bounds give a fixed $\delta>0$ and currents
\[
 T_t\in[\eta_{t,0}]-\delta[\omega]
\]
whose local regularizations satisfy $P_\omega\leq\theta$,
$Q_\omega\leq\Theta$; see 
\cite[Definition~5.3 and Theorem~5.4]{ChuLeeTakahashi}. 
The local regularizations are taken on a fixed finite coordinate
cover, with a range of convolution radius independent of $t$.
The constant $\delta>0$ is also uniform in $t$.

Write $T_t=\eta_{t,0}-\delta\omega
+\sqrt{-1}\partial\bar\partial u_t$, with $\sup_Mu_t=0$.
Since $T_t\geq\cot\Theta\,\omega$ and $\eta_{t,0}$ varies smoothly,
we have $\sqrt{-1}\partial\bar\partial u_t\geq-C\omega$
uniformly for $t$ close to $t_*$.
Choose $t_j\in I$ with $t_j\downarrow t_*$.
After passing to a subsequence, $u_{t_j}\to u$ in $L^1(M)$ and
\(  T_{t_j}\rightharpoonup
  T:=\eta_{t_*,0}-\delta\omega+\sqrt{-1}\partial\bar\partial u
  \in[\eta_{t_*,0}]-\delta[\omega].
\)
For each fixed convolution radius, the regularized local potentials
converge in $C^\infty$ on relatively compact subsets where convolution
is defined. Consequently, the local regularizations of $T$ satisfy
$P_\omega\leq\theta$ and $Q_\omega\leq\Theta$.

Finally choose a sufficiently small
fixed Lelong number  $\epsilon>0$. The set
$E_{\epsilon}=\{x:\nu(T-\cot(\Theta)\omega,x)\geq\epsilon\}$ is a proper
analytic subset. Lemma~\ref{lem:logbar-local-subsolution} gives a smooth
strict subsolution near this set. The regularised maximum construction
of \cite[Lemmas~6.2--6.3 and Section~7]{ChuLeeTakahashi} then gives a
global smooth strict subsolution in $[\eta_{t_*,0}]$.
Lemma~\ref{lem:logbar-twisted-phase} implies $t_*\in I$.
Relative openness forces $t_*=0$, and proves (iii).
\end{proof}

\begin{corollary}
\label{cor:logbar-small-defect}
Suppose that $n\geq3$ and
$G_{p,\theta}(\eta,\omega)>0$ for $1\leq p\leq n-2$. Assume also that
\[
 \int_DG_{n-1}^\theta(\eta,\omega)>0
 \quad\text{for every prime divisor }D,\qquad
 \int_M G_{n-1,\theta}(\eta,\omega)\wedge\omega>0.
 \]
If $\int_MG_{n,\theta}(\eta,\omega)/\int_M\omega^n>
-\epsilon_{n,\theta}$, then $[\eta]$ contains the
solution in Theorem~\ref{thm:logbar-signed-criterion}.
\end{corollary}

\begin{proof}
Consider the test family $\eta+t\omega$. With $G_{0,\theta}=1$, one has
\begin{equation}
\label{eq:logbar-canonical-expansion}
 \int_WG_{p,\theta}(\eta+t\omega,\omega)
 =\sum_{k=0}^p\binom pk t^{p-k}
 \int_WG_{k,\theta}(\eta,\omega)\wedge\omega^{p-k}.
\end{equation}
For $p<n$, all coefficients are nonnegative and the constant term is
positive. For $p=n$, every nonconstant coefficient is nonnegative.
Thus \eqref{eq:logbar-proper-stability} and
\eqref{eq:logbar-ambient-stability} hold, and the theorem applies.
\end{proof}

\begin{samepage}
\subsection{Construction of the logarithmic singular subsolution}

\begin{theorem}
\label{thm:logbar-barrier}
\label{thm:log-subsolution}
Assume $(\chi,\omega)$ satisfies the semisubsolution condition, $N>0$, and let
$Z=\bigcup_{i=1}^ND_i$. There are positive rational numbers $b_i$,
a smooth closed form $\widehat\chi$, and $\varphi\in C^\infty(M,\mathbb R)$
such that
\[
 [\widehat\chi]=[\chi]-\sum_i b_i[D_i],\qquad
 \widehat\chi\in\Gamma_{\omega,\theta,\Theta_0},
\]
and
\begin{equation}
\label{eq:logbar-rho}
 \rho=\varphi+\frac1{2\pi}\sum_i b_i\log\|s_i\|_{h_i}^2
\end{equation}
is quasi-plurisubharmonic, has analytic singularities precisely along
$Z$, and satisfies
\begin{equation}
\label{eq:logbar-current-identity}
 \chi+\ii\partial\bar\partial\rho
 =\widehat\chi+\sum_i b_i[D_i]\geq\widehat\chi.
\end{equation}
On $M\setminus Z$,  $\chi+\ii\partial\bar\partial\rho
 =\widehat\chi$. 
\end{theorem}
\end{samepage}

\begin{proof}
For $N>0$,  we can apply Lemma~\ref{lem:logbar-perturbation} and
Corollary~\ref{cor:logbar-small-defect} apply to $\chi_s$ for a
sufficiently small rational $s>0$, since its normalised top-degree
integral tends to zero from below. They give
$\widehat\chi\in[\chi]-s[E]$ with
$\widehat\chi\in\Gamma_{\omega,\theta,\Theta_0}$.
Let $b_i=sa_i$. The $\partial\bar\partial$-lemma gives a smooth
function $\varphi$ such that
$\widehat\chi=\chi-\sum_i b_i\theta_i+
\ii\partial\bar\partial\varphi$.
Formula~\eqref{eq:logbar-PL} proves
\eqref{eq:logbar-current-identity}.

Choose $q\in\mathbb N$ with $qb_i\in\mathbb N$ for every $i$.
For local defining functions $f_i$ of the divisors,
$\rho=(2\pi q)^{-1}\log|\prod_i f_i^{qb_i}|^2$
up to a smooth function. Thus $\rho$ has the asserted analytic
singularities and is quasi-plurisubharmonic.
\end{proof}

\section{Uniform estimates and locally smooth potential solutions}
\label{sec:bounded-LYZ}

Throughout this section, we work under the assumptions of
Theorem~\ref{thm:regularity of the weak solution}. Recall 
$D_1,\ldots,D_N$ are the destabilizing prime divisors and 
$
 Z=\bigcup_{i=1}^N D_i.
$
By Theorem~\ref{thm:log-subsolution}, after adding a constant to
$\rho$ we may assume that $\sup_M\rho=0$, and there is a smooth
strict subsolution $\widehat\chi$ such that
\begin{equation}\label{eq:section7-barrier}
 \chi+\ii\partial\bar\partial\rho
 =\widehat\chi+\sum_i b_i[D_i]
 \quad\text{on }M,
 \qquad
 \chi+\ii\partial\bar\partial\rho=\widehat\chi
 \quad\text{on }M\setminus Z.
\end{equation}
In particular, $\rho\in C^\infty(M\setminus Z)$ and
$\rho\to-\infty$ along $Z$. By its construction in Section~6, we 
 fix a constant $\Theta_*\in( \theta, \pi)$, such that
\begin{equation}\label{eq:hat-chi-phase}
 P_\omega(\widehat\chi)<\theta,\qquad
 Q_\omega(\widehat\chi)<\Theta_*.
\end{equation}

\subsection{The approximating twisted LYZ equations}

For $t>0$, set $\chi_t:=\chi+t\omega$ and
\begin{equation}\label{eq:ct-LYZ}
 c_t:=\frac{\int_MG_{n,\theta}(\chi_t,\omega)}
 {\int_M\omega^n}.
\end{equation}

\begin{lemma}\label{lem:ct-properties}
There exist constants $c,C,t_0>0$ such that
\begin{equation}\label{eq:ct-linear}
 0<ct\leq c_t\leq Ct,\qquad 0<t\leq t_0.
\end{equation}

\end{lemma}

\begin{proof}
Since
$\int_MG_n(\chi)=0$,
\[
 c_t\int_M\omega^n
 =nt\int_MG_{n-1}(\chi)\wedge\omega
 +\sum_{k=0}^{n-2}\binom nk t^{n-k}
   \int_MG_k(\chi)\wedge\omega^{n-k}.
\]
The semisubsolution condition and
Lemma~\ref{lem:descending-positivity} make all terms on the right
nonnegative. Moreover, Theorem~\ref{thm:modified-nef}, applied to
$[\omega]$, gives
$\int_MG_{n-1}(\chi)\wedge\omega>0$. The conclusion follows
immediately.
\end{proof}

Since $\chi\in\overline\Gamma_{\omega,\theta}$, we have
$\chi+t\omega\in \Gamma_{\omega,\theta}$ for every $t>0$.
Proposition~\ref{prop:twisted-existence} therefore gives, for
$0<t\leq t_0$, a smooth function $u_t$ with
$\sup_Mu_t=0$, such that
\begin{equation}\label{eq:approx-LYZ}
 \eta_t:=\chi+t\omega+\ii\partial\bar\partial u_t
 \in\Gamma_{\omega,\theta},\qquad
 G_{n,\theta}(\eta_t,\omega)=c_t\omega^n.
\end{equation}
If $\lambda_i=\cot\theta_i, 1\le i\le n$ are eigenvalues of $\eta_t$ with
respect to $\omega$, then
\begin{equation}\label{eq:phase-density}
 c_t=
 \frac{\sin\!\left(\theta-\sum_i\theta_i\right)}
 {\sin\theta\prod_i\sin\theta_i}.
\end{equation}
Since $c_t>0$ and $\eta_t\in\Gamma_{\omega,\theta}$, we have
\begin{equation}\label{eq:full-phase-less-theta}
 Q_\omega(\eta_t)=\sum_i\theta_i<\theta.
\end{equation}
In particular, 
\begin{equation}\label{eq:eta-lower}
 \eta_t>\cot\theta\,\omega ,
\end{equation}
which gives a uniform
quasi-plurisubharmonic lower bound for the potentials $u_t$.

\subsection{A local \(C^0\) estimate}

We use the following standard lower-contact-set form of the ABP
estimate.

\begin{lemma}[ABP estimate]\label{lem:ABP-LYZ}
Let $\Omega\subset\mathbb R^N$ be a bounded open set with
$\operatorname{diam}(\Omega)\leq D$, and 
$w\in C^2(\Omega)\cap C^0(\overline\Omega)$. Suppose that
$w(x_0)=m$ and $w\geq m+\varepsilon$ on $\partial\Omega$. If
\[
 \Gamma_w:=\left\{x\in\Omega:\ 
 \exists\,\xi\in B_{\varepsilon/(2D)}(0)
 \text{ such that }
 w(y)\geq w(x)+\xi\cdot(y-x)\ \forall y\in\Omega\right\},
\]
then $D^2_{\mathbb R}w\geq0$,
$w<m+\varepsilon/2$ on $\Gamma_w$, and
\begin{equation}\label{eq:ABP-LYZ}
 |B_1(0)|\left(\frac{\varepsilon}{2D}\right)^N
 \leq\int_{\Gamma_w}\det_{\mathbb R}
       (D^2_{\mathbb R}w)\,dx .
\end{equation}
\end{lemma}

The following lemma will be used for the uniform estimate.
\begin{lemma}\label{lem:pointwise-LYZ}
There exist $\varepsilon_0,R>0$, independent of sufficiently small
$t$, with the following property. Suppose
$\eta\in\Gamma_{\omega,\theta}$ satisfies
\[
 G_{n,\theta}(\eta,\omega)=c_t\omega^n,\qquad
 \eta\geq\widehat\chi+t\omega-\varepsilon_0\omega .
\]
Then
\begin{equation}\label{eq:Y-upper}
 \eta\leq R\omega .
\end{equation}
\end{lemma}

\begin{proof}
Since $\widehat\chi$ is a subsolution, after decreasing
$\varepsilon_0$ and $t_0$ we can find $\delta>0$ such that
\begin{equation}\label{eq:strict-phase-margin}
 P_\omega\bigl(\widehat\chi+(t-\varepsilon_0)\omega\bigr)
 \leq\theta-2\delta,\qquad 0<t\leq t_0.
\end{equation}
The operator $P_\omega$ is decreasing with respect to the Hermitian
matrix order, so the assumed lower bound for $\eta$ gives
$P_\omega(\eta)\leq\theta-2\delta$.

Choose $\omega$-unitary coordinates in which $\eta$ is diagonal and
write its eigenvalues as $\lambda_i=\cot\theta_i$,
$\theta_i\in(0,\pi)$. let
$S_i:=\sum_{j\neq i}\theta_j$. Then
$S_i\leq\theta-2\delta$ for every $i$, while
$\theta(\lambda):=\sum_i\theta_i<\theta$ by~\eqref{eq:phase-density}.

For any $1\le i\le n$. If $\theta_i\geq\delta$, then
$\lambda_i\leq\cot\delta$. If $\theta_i<\delta$, then
\[
 \theta-S=\theta-S_i-\theta_i\geq\delta,
\]
and hence, using~\eqref{eq:phase-density},
\[
 c_t\geq
 \frac{\sin\delta}{\sin\theta\,\sin\phi_i}.
\]
Since $c_t$ is uniformly bounded by Lemma~\ref{lem:ct-properties},
this gives a uniform positive lower bound for $\sin\theta_i$, and
then a uniform upper bound for $\lambda_i$. Thus we have ~\eqref{eq:Y-upper}.
\end{proof}

Then We can prove the local \(C^0\) estimate.
\begin{proposition}
\label{prop:local-C0-LYZ}
There exists a constant $C>0$, independent of $t$, such that
\begin{equation}\label{eq:local-barrier-estimate}
 u_t\geq\rho-C\qquad\text{on }M\setminus Z.
\end{equation}
Consequently, for every $K\Subset M\setminus Z$ there is a uniform 
$C_K>0$, such that
\begin{equation}\label{eq:local-C0-LYZ}
 -C_K\leq u_t\leq0\qquad\text{on }K.
\end{equation}
\end{proposition}

\begin{proof}
Let
$
v_t=u_t-\rho$ on $M\setminus Z$.
By~\eqref{eq:section7-barrier},
\begin{equation}\label{eq:eta-vt}
\eta_t=\widehat\chi+t\omega+\ii\partial\bar\partial v_t
 \qquad\text{on }M\setminus Z,
\end{equation}
and $v_t\to+\infty$ along $Z$.

By~\eqref{eq:eta-lower},
\(
 \ii\partial\bar\partial u_t
 =\eta_t-\chi-t\omega\geq-C\omega.
\)
Since $\sup_Mu_t=0$, the Green function estimate gives
\(
 \int_M|u_t|\,\omega^n\leq C.
\)
As $\rho\in L^1(M)$, we have
\begin{equation}\label{eq:vt-L1}
 \int_M|v_t|\,\omega^n\leq C.
\end{equation}

Let
$
m_t=\inf_{M\setminus Z}v_t.
$
The minimum is attained at $p_t\in M\setminus Z$. Choose
holomorphic coordinates centered at $p_t$ on a ball $B_{2r}$, with
$r>0$ independent of $t$, and define
\[
w_t(z)=v_t(z)+\varepsilon r^{-2}|z|^2,
\]
where $\varepsilon>0$ is fixed and small. Then
\[
w_t(0)=m_t,\qquad
w_t\geq m_t+\varepsilon
\quad\text{on }\partial B_r.
\]
Let $\Omega_t$ be the component containing $0$ of
$
\{w_t<m_t+\varepsilon\}\cap(B_r\setminus Z).
$
Since $v_t\to+\infty$ along $Z$,
\(
w_t\geq m_t+\varepsilon\)
on \(\partial\Omega_t.
\)

Lemma~\ref{lem:ABP-LYZ} gives a set $\Gamma_t\subset\Omega_t$ such that
\begin{equation}\label{eq:ABP-contact-LYZ}
 D^2_{\mathbb R}w_t\geq0,\qquad
 w_t<m_t+\frac{\varepsilon}{2}
 \quad\text{on }\Gamma_t,
\end{equation}
and
\begin{equation}\label{eq:ABP-integral-LYZ}
 c\leq
 \int_{\Gamma_t}
 \det(D^2_{\mathbb R}w_t)\,dx.
\end{equation}
On $\Gamma_t$,
\(
\ii\partial\bar\partial v_t\geq-C\varepsilon\omega.
\)
Choose $\varepsilon$ so that $C\varepsilon\leq\varepsilon_0$.
Lemma~\ref{lem:pointwise-LYZ} and~\eqref{eq:eta-vt} give
\(
\eta_t\leq R\omega\) on \(\Gamma_t.
\)
Thus
\[
0\leq D^2_{\mathbb R}w_t\leq C
\qquad\text{on }\Gamma_t,
\]
and
\[
\det (D^2_{\mathbb R}w_t)\leq C.
\]
Together with~\eqref{eq:ABP-integral-LYZ},
\[
\operatorname{Vol}_\omega(\Gamma_t)\geq c>0.
\]
Since $v_t\leq w_t<m_t+\varepsilon/2$ on $\Gamma_t$, for
$m_t<-\varepsilon$, 
\[
C\geq\int_{\Gamma_t}|v_t|\,\omega^n
\geq
\left(-m_t-\frac{\varepsilon}{2}\right)
\operatorname{Vol}_\omega(\Gamma_t),
\]
while the first inequality follows from ~\eqref{eq:vt-L1}.
Thus $m_t\geq-C$, which gives
\[
u_t\geq\rho-C
\qquad\text{on }M\setminus Z.
\]
Since $\rho$ is bounded below on every compact subset of
$M\setminus Z$,~\eqref{eq:local-C0-LYZ} follows.
\end{proof}

\subsection{A uniform global \(C^0\) estimate}

We first record a quantitative property of the
supercritical cone.

\begin{lemma}\label{lem:quantitative-property}
Let $\chi\in\overline{\Gamma}_{\omega,\theta}$ and
$\widehat\chi\in\Gamma_{\omega,\theta}$ be smooth real $(1,1)$-forms
on $M$, where $0<\theta<\pi$. Set
$
B_s:=(1-s)\chi+s\widehat\chi,\qquad 0\le s\le1.
$
Then $B_s\in\Gamma_{\omega,\theta}$ for $s>0$, and there exists
a constant $\delta>0$ such that
\[
G_{n-1,\theta}(B_s,\omega)\ge\delta s\,\omega^{n-1},
\qquad 0\le s\le1.
\]
\end{lemma}

\begin{proof}
For a complex hyperplane $H\subset T_x^{1,0}M$, write
\(
q_H(A):=Q_{\omega|_H}(A|_H).
\)
Then $P_\omega(A)=\max_H q_H(A)$.
By \cite[Lemma~5.6(9)]{Chen}, with $f=0$ and applied
in dimension $n-1$, the function $A\mapsto\cot q_H(A)$ is
concave on the convex supercritical phase domain
$\{A:q_H(A)<\pi\}$.

Since $P_\omega(\chi)\le\theta$ and
$P_\omega(\widehat\chi)<\theta$, compactness of $M$ gives
a constant $\kappa>0$ such that
\[
\cot P_\omega(\chi)\ge\cot\theta,\qquad
\cot P_\omega(\widehat\chi)\ge\cot\theta+\kappa.
\]
Thus
\[
\begin{aligned}
\cot P_\omega(B_s)
&\ge(1-s)\cot P_\omega(\chi)
+s\cot P_\omega(\widehat\chi)\ge\cot\theta+\kappa s.
\end{aligned}
\]
As $\cot$ is strictly decreasing on $(0,\pi)$, we obtain
$P_\omega(B_s)\le\theta$, with strict inequality for $s>0$.

At a fixed point, choose an $\omega$-unitary frame in which
$B_s$ is diagonal, with eigenvalues $\lambda_1,\ldots,\lambda_n$.
Set
\[
\Theta_i:=\sum_{j\ne i}\arccot\lambda_j,\qquad
R_i:=\prod_{j\ne i}\sqrt{1+\lambda_j^2}.
\]
Since $\Theta_i\le P_\omega(B_s)$, the preceding inequality yields
\(
\cot \Theta_i-\cot\theta\ge\kappa s.
\)
The forms $B_s$, $0\le s\le1$, constitute a smooth compact
family, so their eigenvalues are uniformly bounded.
Hence there exists $\sigma>0$ such that
\(
0<\sigma\le \Theta_i\le\theta<\pi
\)
for all $x\in M$, $s\in[0,1]$, and $1\le i\le n$.
Since $R_i\ge1$, there is a constant $b>0$, independent
of $x,s$, such that $R_i\sin \Theta_i\ge b$. Therefore
\[
\begin{aligned}
&\operatorname{Re}\prod_{j\ne i}(\lambda_j+\sqrt{-1})
-\cot\theta\,
 \operatorname{Im}\prod_{j\ne i}(\lambda_j+\sqrt{-1})
=R_i\sin \Theta_i\,(\cot \Theta_i-\cot\theta)
\ge b\kappa s.
\end{aligned}
\]
Comparing the diagonal coefficients of
$G^\theta_{n-1}(B_s,\omega)$ and $\omega^{n-1}$ gives
\[
G^\theta_{n-1}(B_s,\omega)
\ge b\kappa s\,\omega^{n-1}.
\]
Taking $\delta=b\kappa$ completes the proof.
\end{proof}

Inspired by the trick of Jeffres ~\cite{Jeffres2000}, we use the bounded barrier $\psi\in C^0(M)\cap C^\infty(M\setminus Z) $ defined by
\[
 \psi=e^{a\rho}\quad\text{on }M\setminus Z,\qquad
 \psi|_Z=0,
\]
where $a$ is small positive constant and then 
$0\leq\psi\leq1$.
Since $\rho\leq0$ and $\rho\to-\infty$ along $Z$, the extension
$\psi|_Z=0$ is continuous and $0\leq\psi\leq1$.
\begin{lemma}[Bounded  barrier]
\label{lem:bounded-LYZ-barrier}
For all small $a>0$, 
\begin{equation}\label{eq:tilde-chi-Gamma}
\widetilde\chi:=\chi+\ii\partial\bar\partial\psi
 \in\Gamma_{\omega,\theta}
 \qquad\text{on }M\setminus Z,
\end{equation}
and the function 
\begin{equation}\label{eq:Gn-blowup-LYZ}
 \frac{G_{n,\theta}(\widetilde\chi,\omega)}{\omega^n}(z)
 \longrightarrow+\infty
 \qquad\text{as }z\to Z.
\end{equation}
Moreover, for every smooth function $u$ on $M$, the function
$u-\psi$ cannot attain a local minimum on $Z$.
\end{lemma}
\begin{proof} Since,
on $M\setminus Z$, $\widehat\chi=\chi+\sqrt{-1}\partial\bar\partial \rho$, we have
\begin{equation}\label{eq:tildechi-decomp}
 \widetilde\chi
 =(1-a\psi)\chi+a\psi\widehat\chi
+a^2\psi\,\ii\partial\rho\wedge\bar\partial\rho.
\end{equation}
Then by Lemma \ref{lem:quantitative-property}, 
\eqref{eq:tilde-chi-Gamma} follows. By Lemma \ref{lem:quantitative-property}, we also have
\begin{align*}
G_{n,\theta}(\widetilde\chi,\omega)
 =&G_{n,\theta}(B_\lambda,\omega)
   +nR\wedge G_{n-1,\theta}(B_\lambda,\omega)\\
 \geq&-C\omega^n+\delta a^3 e^{2a\rho}
|\partial\rho|_\omega^2\omega^n.
\end{align*}
We next show
\begin{equation}\label{eq:rho-gradient-blowup}
 e^{2a\rho}|\partial\rho|_\omega^2\longrightarrow+\infty
 \qquad\text{along }Z.
\end{equation}

Let $\pi:\widetilde M\to M$ be a log resolution of $Z$ which is an
isomorphism over $M\setminus Z$, and fix a K\"ahler metric
$\widetilde\omega$ on $\widetilde M$. Around a point of
$\pi^{-1}(Z)$ choose coordinates
$w=(w_1,\ldots,w_n)$ such that
$\pi^{-1}(Z)=\{w_1\cdots w_k=0\}$. Since $\rho$ has logarithmic
singularities with strictly positive divisorial coefficients,
\begin{equation}\label{eq:rho-resolution-LYZ}
 \pi^*\rho=\sum_{j=1}^k b_j\log|w_j|^2+H,\qquad b_j>0,
\end{equation}
where $H$ is smooth. Since $\pi^*\omega\leq C\widetilde\omega$,
\[
 |\partial\rho|_\omega^2\circ\pi
 \geq C^{-1}|\partial(\pi^*\rho)|_{\widetilde\omega}^2.
\]
Set $d(w):=\min_{1\leq j\leq k}|w_j|$. After shrinking the chart,
\[
 |\partial(\pi^*\rho)|_{\widetilde\omega}^2\geq c\,d(w)^{-2},
 \qquad
 e^{2a\pi^*\rho}\geq c\,d(w)^{4aB},
 \quad B:=\sum_{j=1}^k b_j.
\]
Thus
\[
 \bigl(e^{2a\rho}|\partial\rho|_\omega^2\bigr)\circ\pi
 \geq c\,d(w)^{4aB-2}.
\]
Using finitely many such charts and choosing $a>0$ so small that
$2aB<1$ on each of charts proves~\eqref{eq:rho-gradient-blowup}, and
hence~\eqref{eq:Gn-blowup-LYZ}.

It remains to prove that for every smooth function $u$ on $M$, the function
$u-\psi$ cannot attain a local minimum on $Z$. Suppose that $u-\psi$ has a
local minimum at $p\in Z$, choose $\widetilde p\in\pi^{-1}(p)$ and
an SNC chart $U$ near $\widetilde p$ as above. Along the holomorphic disc
\[
 \gamma:\Delta\rightarrow U, \quad\gamma(t)=(t,\ldots,t,0,\ldots,0),
\]
with the first $k$ coordinates equal to $t$, we have
$\pi^*\rho(\gamma(t))=B\log|t|^2+O(1)$, and thus
$\psi(\pi(\gamma(t)))\geq c|t|^{2aB}$. Then
\[
 u(\pi(\gamma(t)))-u(p)\geq c|t|^{2aB},
\]
whereas smoothness of $u$ gives an upper bound
$C|t|$. Since $2aB<1$, this is impossible as $t\to0$. The lemma is
proved.
\end{proof}

Then we get the uniform $L^{\infty}$ estimate of $u_t$.
\begin{proposition}
\label{prop:global-C0-LYZ}
There exists $C>0$, independent of $0<t\leq t_0$, such that
\begin{equation}\label{eq:global-C0-LYZ}
 -C\leq u_t\leq0\qquad\text{on }M.
\end{equation}
\end{proposition}

\begin{proof}
Let $\psi$ be given by Lemma~\ref{lem:bounded-LYZ-barrier}. By
Lemma~\ref{lem:ct-properties},
$
 \sup\limits_{0<t\leq t_0}c_t\le C
$.
By ~\eqref{eq:Gn-blowup-LYZ}, there exists a   open set $U\supset Z$ such that
\begin{equation}\label{eq:barrier-large-Gn}
G_{n,\theta}(\widetilde\chi,\omega)>2C\omega^n
 \qquad\text{on }U\setminus Z.
\end{equation}
Set $K:=M\setminus U\subset \subset M\setminus Z$.
Assume  $u_t-\psi$ attains its minimum at 
$x_t\in M$. By Lemma~\ref{lem:bounded-LYZ-barrier}, $x_t\notin Z$.
At $x_t$,
\[
 H_t:=t\omega+\ii\partial\bar\partial(u_t-\psi)\geq0,
\]
and $\eta_t=\widetilde\chi+H_t$.
For $0\leq s\leq1$, we have $\widetilde\chi+sH_t\in 
\Gamma_{\omega,\theta}$ by~\eqref{eq:tilde-chi-Gamma} and then
\[
 \frac{d}{ds}G_{n,\theta}(\widetilde\chi+sH_t,\omega)
 =nG_{n-1}^\theta(\widetilde\chi+sH_t,\omega)\wedge H_t\geq0.
\]
Hence
\begin{equation}\label{eq:Gn-monotone}
 G_{n,\theta}(\eta_t,\omega)
 \geq G_{n,\theta}(\widetilde\chi,\omega)
 \qquad\text{at }x_t.
\end{equation}
Then, if $x_t\in U\setminus Z$, we have
\[
 c_t\omega^n=G_{n,\theta}(\eta_t,\omega)
 \geq G_{n,\theta}(\widetilde\chi,\omega)
 >2C\omega^n,
\]
which is a contradiction. Thus $x_t\in K$, for every $0<t\leq t_0$.

By Proposition~\ref{prop:local-C0-LYZ},
$u_t(x_t)\geq-C_K$. Since $0\leq\psi\leq1$,
$u_t(x_t)-\psi(x_t)\geq-C_K-1$. Then
$u_t-\psi\geq-C_K-1$ on $M$, and thus $u_t\geq-C_K-1$.
Together with $\sup_Mu_t=0$, this proves~\eqref{eq:global-C0-LYZ}.
\end{proof}

\subsection{Local higher order estimates}
\label{subsec:weighted-C2}
On $M\setminus Z$, recall that
$\eta_t=\beta_t+\ii\partial\bar\partial v_t$, where
$\beta_t=\widehat\chi+t\omega$ and $v_t=u_t-\rho$.
We have $u_t\leq0$, $v_t\geq-C$, and $v_t\to+\infty$ along $Z$.

Fix $\theta<\Theta_*<\Theta<\pi$ with
$Q_\omega(\widehat\chi)<\Theta_*$.
Since $\widehat\chi$ is a smooth strict subsolution, we can choose
a fixed $\delta>0$ such that
\begin{equation}\label{eq:strict-subsolution}
 P_\omega(\beta_t)\leq\theta-\delta,
 \qquad Q_\omega(\beta_t)\leq\Theta_*,
\end{equation}
and $\beta_t-2\delta\omega\in\Gamma_{\omega,\theta,\Theta}$
for all sufficiently small $t>0$.
All constants below are independent of $t$.

Set
\begin{equation}\label{eq:Ft}
 F_t(\eta)=
 \frac{\Rea(\eta+\ii\omega)^n-c_t\omega^n}
      {\Ima(\eta+\ii\omega)^n}-\cot\theta.
\end{equation}
Then $F_t(\eta_t)=0$.
By \cite[Lemma~5.6(3), (9), (10)]{Chen}, $F_t$ is elliptic
and concave on the convex set $\Gamma_{\omega,\theta,\Theta}$.
Write $w_{i\bar j}=(\eta_t)_{i\bar j}$,
$F_t^{i\bar j}=\partial F_t/\partial w_{i\bar j}$, and
$\mathcal L_t=F_t^{i\bar j}\nabla_i\nabla_{\bar j}$.
All derivatives are covariant derivatives with respect to $\omega$.

Let $\lambda_1\geq\cdots\geq\lambda_n$ be the
eigenvalues of $\eta_t$ with respect to $\omega$, and write $q=Q_\omega(\eta_t)$.
The identity
$c_t=\sin(\theta-q)\prod_i\sqrt{1+\lambda_i^2}/\sin\theta$
and $c_t\to0$ give
\begin{equation}\label{eq:small-twist}
 \theta/2\leq q<\theta,\qquad
 (\theta-q)(1+\max_i|\lambda_i|)\leq Cc_t.
\end{equation}
In an $\omega$-unitary frame where  $\eta_t$ is diagonal, direct
differentiation gives $F_t^{i\bar j}=F_i\delta_{ij}$, where
\begin{equation}\label{eq:Fi}
 F_i=
 \frac{\cos(\theta-q)+\lambda_i\sin(\theta-q)}
 {\sin\theta\sin q\,(1+\lambda_i^2)},
 \qquad
 \frac{C^{-1}}{1+\lambda_i^2}\leq F_i
 \leq\frac{C}{1+\lambda_i^2}.
\end{equation}

We apply the subsolution argument of  Sz\'ekelyhidi~\cite[Proposition~6]{Sze2018}
to the  subsolution $\beta_t$.

\begin{lemma}\label{lem:separation}
There are constants $C_0,c_0>0$, independent of $t$, such that
$-\mathcal L_t v_t\geq c_0$ at every point of $M\setminus Z$
where $\lambda_1(\eta_t)\geq C_0$.
\end{lemma}

\begin{proof}
Since $\beta_t-2\delta\omega\in\Gamma_{\omega,\theta,\Theta}$,
the proof of \cite[Lemma~5.6(6)]{Chen} gives a fixed $R>0$ such that
\[
 \bigl(\lambda(\beta_t)-2\delta\mathbf 1+\mathbb R_{>0}^n\bigr)
 \cap\{F_t=0\}\subset B_R(0).
\]
Here $\mathbf 1=(1,\ldots,1)$, and the level set is taken in
$\Gamma_{\theta,\Theta}$.

The proof of \cite[Proposition~6]{Sze2018} then gives
$\kappa>0$ such that, whenever $\lambda_1>R$, either
\[
 -\mathcal L_t v_t\geq\kappa\sum_iF_i,
 \qquad\text{or}\qquad
 F_i\geq\kappa\sum_jF_j\quad\text{for every }i.
\]
The constant $\kappa$ is independent of $t$.
Indeed, the sets above lie in a fixed ball.
On these sets, $F_t\to F_0$ smoothly, and \eqref{eq:Fi}
gives a fixed positive lower bound for each $F_i$.

By \eqref{eq:small-twist},
$\cot\theta<\lambda_n\leq\cot(\theta/(2n))$.
Thus \eqref{eq:Fi} gives fixed constants $c,C>0$ such that
\[
 \sum_iF_i\geq c,
 \qquad F_1\leq\frac{C}{1+\lambda_1^2}.
\]
Set $c_0=\kappa c$.
Choose $C_0>R$ such that $C/(1+C_0^2)<c_0$.
If $\lambda_1\geq C_0$, the second alternative would give
$F_1\geq\kappa\sum_iF_i\geq c_0$, whereas $F_1<c_0$.
Hence the first alternative holds, and
$-\mathcal L_t v_t\geq\kappa\sum_iF_i\geq c_0$.
\end{proof}

\begin{proposition}\label{prop:C2}
For every $K\subset\subset M\setminus Z$, there is a constant $C_K$,
independent of $t$, such that $\Delta_\omega u_t\leq C_K$ on $K$.
\end{proposition}

\begin{proof}
For a large constant $A$, set
\[
W=\tr_\omega(\eta_t+A\omega)>1.
\]
We first prove
\[
\mathcal L_t\log W\geq-C,
\]
where $C$ is independent of $A$ and $t$. Choose normal holomorphic
coordinates in which
$w_{i\bar j}=\lambda_i\delta_{ij}$ at a point.
Differentiating the equation gives
\begin{equation}\label{eq:first-covariant}
 \sum_iF_iw_{i\bar i,k}=0.
\end{equation}
A second differentiation, followed by commuting derivatives, gives
\begin{equation}\label{eq:LW}
 \mathcal L_tW\geq
 -\sum_{i,j,k}(\nabla_{\bar k}F_t^{i\bar j})w_{i\bar j,k}-CW.
\end{equation}
Here we used $d\eta_t=0$ and
$\sum_iF_i(1+|\lambda_i|)\leq C$ to control the curvature terms.

For each fixed $k$, differentiating $F_t$ and using
\eqref{eq:first-covariant} to cancel the terms arising from the
derivative of the denominator, we obtain
\begin{equation}\label{eq:third-order-pre}
 -\sum_{i,j}(\nabla_{\bar k}F_t^{i\bar j})w_{i\bar j,k}
 =\sum_{i,j}a_{ij}
 \bigl(|w_{i\bar j,k}|^2-w_{i\bar i,k}w_{j\bar j,\bar k}\bigr),
\end{equation}
where
\[
 a_{ij}=
 \frac{(\lambda_i+\lambda_j)\cos(\theta-q)
 +(\lambda_i\lambda_j-1)\sin(\theta-q)}
 {\sin\theta\sin q\,(1+\lambda_i^2)(1+\lambda_j^2)}.
\]
Moreover, \eqref{eq:first-covariant} is equivalently
\[
 \cos(\theta-q)\sum_i\frac{w_{i\bar i,k}}{1+\lambda_i^2}
 +\sin(\theta-q)\sum_i
 \frac{\lambda_iw_{i\bar i,k}}{1+\lambda_i^2}=0.
\]
Substituting this identity into the second term on the right-hand side
of \eqref{eq:third-order-pre}, we get
\begin{align}\label{eq:third-order-identity}
 -\sum_{i,j}(\nabla_{\bar k}F_t^{i\bar j})w_{i\bar j,k}
 &=\sum_{i,j}a_{ij}|w_{i\bar j,k}|^2\notag\\
 &\quad+
 \frac{\sin(\theta-q)}
 {\sin\theta\sin q\cos^2(\theta-q)}
 \left|
 \sum_i\frac{\lambda_iw_{i\bar i,k}}{1+\lambda_i^2}
 \right|^2.
\end{align}

We claim that, for $A$ sufficiently large and $t$ sufficiently small,
\begin{equation}\label{eq:key-inequality}
 \sum_{i,j}a_{ij}|w_{i\bar j,k}|^2
 \geq
 \sum_{i,j}\frac{F_i}{\lambda_j+A}|w_{i\bar j,k}|^2.
\end{equation}
For $i\ne j$, put $\theta_i=\arccot\lambda_i$. Since
\[
 0<\theta-q+\theta_i\leq\theta-\theta_j<\pi,
\]
and $\cot$ is decreasing on $(0,\pi)$, we have
\[
 \frac{a_{ij}}{F_i}
 =\frac{\lambda_j+\cot(\theta-q+\theta_i)}{1+\lambda_j^2}
 \geq
 \frac1{\lambda_j-\cot\theta}
 \geq
 \frac1{\lambda_j+A},
\]
provided $A$ is large enough.

It remains to consider the diagonal terms. Set
\[
 d_i=\frac{a_{ii}-F_i/(\lambda_i+A)}{F_i^2}.
\]
A direct computation gives
\begin{align*}
 d_i
 ={}&\frac{\sin\theta\sin q}
 {\bigl(\cos(\theta-q)+\lambda_i\sin(\theta-q)\bigr)^2}
 \\
 &\times\left[
 2\lambda_i\cos(\theta-q)
 +(\lambda_i^2-1)\sin(\theta-q)
 -\frac{(1+\lambda_i^2)
 \bigl(\cos(\theta-q)+\lambda_i\sin(\theta-q)\bigr)}
 {\lambda_i+A}
 \right].
\end{align*}
We have $\lambda_i>\cot\frac{\theta}{2}>0$ for $i<n$, while
$\cot\theta<\lambda_n\leq\cot\frac{\theta}{2n}$. Hence, by
\eqref{eq:small-twist},
\begin{align}\label{eq:diagonal-coefficients}
 d_i>0,\qquad
 d_i^{-1}
 &=\frac1{2\sin^2\theta\,\lambda_i}
 +O(A^{-1}+c_t),\qquad i<n,\notag\\
 d_n
 &=2\sin^2\theta\,\lambda_n+O(A^{-1}+c_t).
\end{align}
Indeed, for $i<n$,
\[
 d_i
 =\sin^2\theta
 \left(2\lambda_i-\frac{1+\lambda_i^2}{\lambda_i+A}\right)
 +O(c_t\lambda_i),
\]
and, for $A$ sufficiently large,
\[
 2\lambda_i-\frac{1+\lambda_i^2}{\lambda_i+A}\geq\lambda_i.
\]
Moreover, since $\lambda_i>\cot\frac{\theta}{2}$,
\[
 \frac{1}{
 \sin^2\theta\left(
 2\lambda_i-\frac{1+\lambda_i^2}{\lambda_i+A}
 \right)}
 -\frac1{2\sin^2\theta\,\lambda_i}
 =
 \frac{1+\lambda_i^2}
 {2\sin^2\theta\,\lambda_i
 (\lambda_i^2+2A\lambda_i-1)}
 =O(A^{-1}).
\]
uniformly for $i<n$. This proves the first line of
\eqref{eq:diagonal-coefficients}; the second follows directly from the
uniform bound for $\lambda_n$.

If $\lambda_n<0$, then
\[
 \sum_{i<n}\lambda_i^{-1}
 \leq\tan(\theta-\arccot\lambda_n),
\]
and thus
\[
 1+\lambda_n\sum_{i<n}\lambda_i^{-1}
 \geq
 \frac{\sin\theta(1+\lambda_n^2)}
 {\sin\theta+\lambda_n\cos\theta}
 \geq\sin\theta.
\]
If $\lambda_n\geq0$, the same lower bound is immediate. Thus
\eqref{eq:diagonal-coefficients} gives
\[
 1+d_n\sum_{i<n}d_i^{-1}
 =1+\lambda_n\sum_{i<n}\lambda_i^{-1}
 +O(A^{-1}+c_t)>0,
\]
after first fixing $A$ sufficiently large and then taking $t$
sufficiently small. By \eqref{eq:first-covariant} and the
Cauchy--Schwarz inquality,
\[
 \sum_i d_i|F_iw_{i\bar i,k}|^2
 \geq
 \left(d_n+\frac1{\sum_{i<n}d_i^{-1}}\right)
 |F_nw_{n\bar n,k}|^2
 \geq0.
\]
This proves \eqref{eq:key-inequality}.

Since $d\eta_t=0$, we have
$w_{i\bar i,k}=w_{k\bar i,i}$. Hence, by Cauchy--Schwarz,
\[
 \sum_kF_k|W_k|^2
 \leq
 W\sum_{i,k}\frac{F_k}{\lambda_i+A}|w_{k\bar i,i}|^2
 \leq
 W\sum_{i,j,k}\frac{F_i}{\lambda_j+A}|w_{i\bar j,k}|^2.
\]
Combining this with \eqref{eq:LW},
\eqref{eq:third-order-identity}, and \eqref{eq:key-inequality}, we get
\[
 \mathcal L_tW\geq\frac1W\sum_kF_k|W_k|^2-CW,
\]
and therefore
\[
 \mathcal L_t\log W
 =\frac{\mathcal L_tW}{W}
 -\frac1{W^2}\sum_kF_k|W_k|^2
 \geq-C.
\]

Since
$v_t\to+\infty$ along $Z$, the function
$\log W-Av_t$ tends to $-\infty$ along $Z$ and hence attains its
maximum at some $x_t\in M\setminus Z$. If
$\lambda_1(x_t)\geq C_0$, Lemma~\ref{lem:separation} gives
\[
 0\geq
 \mathcal L_t(\log W-Av_t)(x_t)
 \geq-C+Ac_0,
\]
which is impossible when $A$ is chosen sufficiently large. Thus
$\lambda_1(x_t)<C_0$, and hence
\[
 W(x_t)\leq n(C_0+A).
\]
Thus we have
\[
 W(x)
 \leq n(C_0+A)e^{A(v_t(x)-v_t(x_t))}
 \leq Ce^{-A\rho(x)},
\]
where we used $u_t\leq0$ and $v_t\geq-C$. Since $\rho$ is bounded
on $K$ and
\[
 \Delta_\omega u_t
 =W-nA-\tr_\omega\chi-nt,
\]
the desired estimate follows.
\end{proof}

The trace bound and $\eta_t>\cot\theta\,\omega$ bound all
eigenvalues on compact subsets of $M\setminus Z$.
By \eqref{eq:Fi}, the equation is uniformly elliptic there.
The  $C^0$ and Laplace bounds give local $W^{2,p}$ bounds
for every finite $p$, hence local gradient bounds.
Concavity and the complex Evans--Krylov estimate then give local
$C^{2,\alpha}$ bounds. Differentiating the equation and applying
the interior Schauder estimates yields
$\|u_t\|_{C^k(K,\omega)}\leq C_{k,K}$ for every $k\geq0$.

\subsection{Proof of Theorem~1.6}
\label{subsec:convergence}

\begin{proof}[Proof of Theorem~\ref{thm:regularity of the weak solution}]
Let $t_j>0$ with $t_j\to0$. Since $\eta_t>\cot\theta\,\omega$,
there is a constant $A>0$, independent of $t$, such that
$\ii\partial\bar\partial u_t\geq-A\omega$.
By the compactness theorem for quasi-plurisubharmonic functions
and the normalization $\sup_Mu_t=0$, after passing to a subsequence,
$u_{t_j}$ converges in $L^1(M)$ to a quasi-plurisubharmonic function
$u_\infty$. The uniform $C^0$ estimate gives
$\|u_\infty\|_{L^\infty(M)}\leq C$, and Hartogs' lemma gives
$\sup_Mu_\infty=0$.
By the local higher order estimates, after passing to a further
subsequence, we have
\[
u_{t_j}\longrightarrow u_\infty
\qquad\text{in }C^\infty_{\mathrm{loc}}(M\setminus Z).
\]

The $L^1$ convergence also gives
\begin{equation}\label{eq:S-infinity-positive}
\eta_\infty
:=\chi+\ii\partial\bar\partial u_\infty
\geq\cot\theta\,\omega
\end{equation}
in the sense of currents.
Set $S_\infty:=\eta_\infty-\cot\theta\,\omega$.
Then $S_\infty$ is a closed positive $(1,1)$-current with
locally bounded potentials. By Bedford--Taylor theory
\cite{BT1982}, the measures
$S_\infty^k\wedge\omega^{n-k}$, $0\leq k\leq n$,
are well defined and do not charge the proper analytic subset $Z$.
Using $\eta_\infty=S_\infty+\cot\theta\,\omega$,
we define $G_{n,\theta}(\eta_\infty,\omega)$ by expansion as a
real linear combination of these measures.

Since $c_{t_j}\to0$, the smooth convergence on $M\setminus Z$
and the equations
$G_{n,\theta}(\eta_{t_j},\omega)=c_{t_j}\omega^n$ give
\[
G_{n,\theta}(\eta_\infty,\omega)=0
\qquad\text{on }M\setminus Z.
\]
Moreover, the Bedford--Taylor measures appearing in the expansion of
$G_{n,\theta}(\eta_\infty,\omega)$ do not charge the analytic set $Z$,
the equation holds on $M$ in the Bedford--Taylor sense.
\end{proof}

\end{document}